\documentclass[a4paper, 11pt]{article}

\usepackage{authblk}
\usepackage{comment}
\usepackage{amssymb}
\usepackage{enumerate}
\usepackage{amsthm}
\usepackage{amscd}
\usepackage{graphicx}
\usepackage{titlesec}
\usepackage{indentfirst}
\titleformat{\section}{\bfseries}{\thesection .}{0.5em}{}
\titleformat{\subsection}{\itshape}{\thesubsection .}{0.5em}{}
\usepackage[top=25.4mm, bottom=25.4mm, left=27.7mm, right=27.8mm]{geometry}

\usepackage[tbtags]{amsmath}
\usepackage{booktabs}
\allowdisplaybreaks  

\usepackage{caption}
\begin{document}

\theoremstyle{plain}
\newtheorem{theorem}{Theorem}[section]
\newtheorem{lemma}[theorem]{Lemma}
\newtheorem{corollary}[theorem]{Corollary}
\newtheorem{proposition}[theorem]{Proposition}
\newtheorem{assumption}[theorem]{Assumption}
\newtheorem{problem}[theorem]{Problem}

\theoremstyle{definition}
\newtheorem{definition}[theorem]{Definition}
\newtheorem{example}[theorem]{Example}

\theoremstyle{remark}
\newtheorem{remark}[theorem]{Remark}
\newtheorem{notation}{Notation}


\title{\Large  Optimization of conditional convex risk measures }
\author{
{\normalsize Tiexin Guo \thanks{
Email addresses: tiexinguo@csu.edu.cn } }\\
\textit{\small School of Mathematics and Statistics, Central South University, Changsha 410083, China.}}
\date{}
\maketitle

\renewcommand{\baselinestretch}{1.2}
\large\normalsize
\noindent \rule[0.5pt]{14.5cm}{0.6pt}\\
\noindent
\textbf{Abstract}\\
Optimization of conditional convex risk measure is a central theme in dynamic portfolio selection theory, which has not yet systematically studied in the previous literature perhaps since conditional convex risk measures are neither random strictly convex nor random coercive. The purpose of this paper is to give some basic results on the existence and uniqueness on this theme, in particular our results for conditional monotone mean--variance and conditional entropic risk measures are complete and deep. As the basis for the work of this paper, this paper first begins with a brief introduction to random functional analysis, including the historical backgrounds for its birth and some important advances. This paper then further surveys some recent progress in random convex analysis and its applications to conditional convex risk measures. Finally, based on these, we establish a concise sufficient and necessary condition for a return to be a solution to the optimization problem of conditional monotone mean--variance. We also make use of the recently developed theory of $L^0$--convex compactness to establish the existence of the optimization problem of conditional entropic risk measure when the conditional mean of returns is given and the returns fall within a random closed ball. Besides, the related uniqueness problems are also solved.

\vspace{0.3cm}

\noindent \rule[0.5pt]{14.5cm}{0.6pt}
\section{Introduction}
Random functional analysis is functional analysis based on random metric spaces, random normed modules, random inner product modules and random locally convex modules, which was born in the course of the development of the theory of probabilistic metric spaces. Probabilistic metric spaces (briefly, $PM$ spaces) were initiated by K.Menger and founded by B.Schweizer and A.Sklar, of which the main idea is to use a probability distribution function to define the probabilistic distance between any given two points, see \cite{SS83} for the fundamentals and historical survey for the theory of $PM$ spaces. Similarly, A.N.\v{S}erstn\v{e}v introduced in 1962 the notion of a probabilistic normed space (briefly, a $PN$ space) as a probabilistic generalization of an ordinary normed space. Therefore, Menger's theory of $PM$ spaces initiated the idea of randomizing space theory of functional analysis. Naturally, one may also consider the randomization of traditional space theory by starting from the standard measure--theoretic model of probability theory, which leads A.\v{S}pac\v{e}k to formulate the notion of an uniform random metric space, subsequently the notions of a general random metric space (briefly, an $RM$ space) and random normed space (briefly, an $RN$ space) were introduced and studied in \cite[Chapters 9 and 15]{SS83}, whose main idea is to use a nonnegative random variable to express the random distance between two points or a random norm of a vector. It is well known that the theory of $PM$ spaces had been deeply developed before 1983, the theory of $PN$ spaces, however, has not obtained any substantial advance up to now. This is mainly because $PN$ spaces are not locally convex in general under the traditional $(\varepsilon,\lambda)$--topology so that the theory of conjugate spaces fails in the study of $PN$ spaces.

Our breakthrough began with the study of $RN$ spaces in 1989. Compared with general $PN$ spaces, $RN$ spaces have the additional measure--theoretic structure and stronger geometric structure, by means of which Guo introduced the notion of an almost surely (briefly, a.s.) bounded random linear functional and proved the Hahn--Banach theorem for such random linear functionals in \cite{Guo89}, which leads to the development of the theory of random conjugate spaces for $RN$ spaces. Further, Guo introduced the notions of a random normed module (briefly, an $RN$ module) and random inner product module (briefly, an $RIP$ module) in \cite{YZG91,Guo92,Guo93}, which makes the theory of random conjugate spaces obtain a systematic and deep development, for example, the representation theorems of random conjugate spaces for concrete and important $RN$ modules were obtained in \cite{Guo96,GY96}, even the James theorem characterizing random reflexivity was established in \cite{GL05}, and the closely related issues were also thoroughly studied in \cite{Guo08}. Here, we would like to emphasize that it is the module structure of an $RN$ module that has played a crucial role in these developments. Subsequently, the notion of a random locally convex module (briefly, an $RLC$ module), as a random generalization of a locally convex space, was introduced by Guo in \cite{Guo01} for the systematic development of the theory of random conjugate spaces so that a separation theorem between a point and a closed $L^0$--convex set was established in \cite{GXC09} and the theory of random duality was also deeply developed in \cite{GC09}. It should be pointed out that the pre--2009 work on random functional analysis was carried out under the $(\varepsilon,\lambda)$--topology, here we should also mention the following two famous works whcih are independent of the above--stated work on random functional analysis: the first is Hansen and Richard's work \cite{HR87} where they established the Riesz representation theorem on a class of complete random inner product modules (namely conditional $L^2$--spaces) which played a crucial role in representing a conditional pricing function; the second is Haydon, Levy and Raynaud's work \cite{HLR91} where they also introduced the notion of an $RN$ module (called a randomly normed $L^0$--module in terms of \cite{HLR91}) as a tool for the study of ultrapowers of Lebesgue--Bochner function spaces.

The theory of risk measures occupies a central place in quantative risk management. At the early stage of measuring risk, variance and $VaR$ were employed as risk measures in \cite{M52,J01}, but variance is not monotone and $VaR$ is not subadditive, their applications would lead to some unreasonable conclusions, which also motivates an axiomatic study of risk measures in \cite{ADEH99} where the notion of a coherent risk measure was introduced. Subsequently, a more general convex risk measure was introduced in \cite{FS11,FRG02}. It is well known that classical convex analysis \cite{ET99,P89} is a foundation for convex risk measures.

In multiperiod or dynamic financial mathematics, the notion of a conditional or dynamic convex risk measure was introduced in \cite{DS05,B04,FKV12} in order to precisely measure risk by means of information available from financial markets. At first, classical convex analysis could still deal with conditional convex risk measures on the model space of bounded financial positions \cite{DS05,B04}, but classical convex analysis does not apply to conditional convex risk measures on the model space of unbounded financial positions. This motivated Filipovi\'{c}, Kupper and Vogelpoth in \cite{FKV09} to make an attempt to establish a kind of convex analysis over a locally $L^0$--convex module (called $L^0$--convex analysis in \cite{FKV09}) in order to provide a new analytic foundation for the study of conditional convex risk measures, in particular their work also led to another kind of topology--the locally $L^0$--convex topology for random locally convex modules. Thus, Filipovi\'{c} et.al's work \cite{FKV09} is, without doubt, important, whose importance lies in establishing, for the first time, the link between random functional analysis and the theory of conditional convex risk measures, and also providing another approach to the development of random functional analysis under the locally $L^0$--convex topology.

Following \cite{FKV09}, Guo \cite{Guo10} was devoted to establishing some inherent connections between the two theories derived from the two topologies (namely, the $(\varepsilon,\lambda)$--topology and the locally $L^0$--convex topology) for a random locally convex module. The success of \cite{Guo10} lies in introducing the notion of the countable concatenation property (also called $\sigma$--stability) for a subset of an $L^0$--module, which has played a crucial role in the process of establishing the two theories. Since the paper \cite{FKV09} did not realize the importance of the algebraic notion of $\sigma$--stability, there are some serious defects in $L^0$--convex analysis established in \cite{FKV09}, see \cite{WG15,Zap17} for details. To lay a solid analytic foundation for the theory of conditional risk measures, Guo et.al \cite{GZZ15a,GZZ15b,GZWYYZ17} reestablished the $L^0$--convex analysis by choosing random locally convex modules rather than locally $L^0$--convex modules as the space framework and simultaneously considering the two kinds of topologies. Convex analysis established in such a way is naturally called random convex analysis in accordance with the idea of random functional analysis, at the same time Guo et.al \cite{GZZ14} unified various kinds of conditional convex risk measures presented in \cite{FKV12} together with their conditional convex dual representation theorems. Further, based on the work and the notion of $\sigma$--stability of \cite{Guo10}, a series of important developments can be carried out in \cite{CHKP16,CKV15,DKKS13,FM14,Guo13,GS11,GXC09,GZWG19,GZWW20,GZWYYZ17,GZZ14,GZZ15a,GZZ15b,Wu12,Wu13,WG15,Zap17}.

Although variance, as a risk measure, has its own shortcomings, it also possesses some good behaviors, for example, it is a coercive and strictly convex continuous function on the constraint set in problem (a closed convex set of the Hilbert space $L^2$) so that the existence and uniqueness on the optimization problem are both not a problem. To overcome the shortcomings of variance, people considers replacing it with the semivariance in \cite{JYZ05} and the best risk measures--convex risk measures in portfolio selection in \cite{MMRT09,RS06} and the references therein. However, convex risk measures are neither strictly convex nor coercive, which makes the existence and uniqueness a delicate and complicated problem. According to the fundamental theorem of convex optimization \cite{ET99}: let $(E,\mathcal{T})$ be a Hausdorff locally convex space, $C$ a weakly compact convex subset of $E$ and $f : C \to (-\infty, +\infty]$ a proper lower semicontinuous convex function, then there exists $x \in C$ such that $f(x) = \min \{ f(y) : y \in C \}$, the optimization problem of convex risk measures are within convex optimization, but it needs a different treatment since convex risk measures have their own special properties. On the other hand, finance and economy also forces people to consider the optimization problems of convex functions on not locally convex spaces, at this time it makes no sense to speak of weak compactness for a convex set in a not locally convex space, Gordon \v{Z}itkovi\'{c} introduced in \cite{Zit10} the notion of convex compactness for a convex set in a not locally convex space, which better solved the corresponding optimization problem.

The central purpose of this paper is to study the optimization problem of conditional convex risk measures (including convex risk measures as a special case). Also, optimization of conditional convex risk measures is to prepare for the study of continuous--time conditional mean--conditional convex risk measure portfolio selection, see \cite{GZWG19} for details. Hansen and Richard \cite{HR87} is the first to give the conditional version of Markowitz mean--variance analysis, namely consider the optimization problem of conditional variance when the conditional mean of returns is given. Since the conditional variance has the properties similar to variance, the existence and uniqueness are both not a problem, the really nontrivial contribution of \cite{HR87} is to develop the theory of a class of conditional Hilbert spaces (namely, conditional $L^2$--spaces), this class of conditional $L^2$--spaces are a class of special complete $RIP$ modules. The essence of our work of this paper is to consider the same optimization problem as in \cite{HR87} when the conditional variance used in \cite{HR87} is replaced by conditional convex risk measures. In fact, Hansen and Richard \cite{HR87} is also the first to study optimization of $L^0$--convex functions on closed $L^0$--convex sets. Since complete $RN$ and $RIP$ modules are not locally convex spaces in general under the $(\varepsilon,\lambda)$--topology, motivated by the work \cite{Zit10} Guo et.al \cite{GZWW20} had recently developed the theory of $L^0$--convex compactness for closed $L^0$--convex subsets in complete $RN$ modules, as applications we gave the fundamental theorem of random convex optimization and generalized the optimization results of conditional variance of \cite{HR87}, see \cite{GZWW20} for details. In particular, the characterization theorem for a closed $L^0$--convex subset of a complete $RN$ module to have $L^0$--convex compactness, which was established in \cite{GZWW20} by using the theory of random conjugate spaces, will play a crucial role in this paper as well as in \cite{GZWW20}. Precisely, this paper first establishes a concise sufficient and necessary condition for a return to be a solution to the optimization problem of conditional monotone mean--variance, and then makes use of the recently developed theory of $L^0$--convex compactness to obtain an existence result of the optimization problem of conditional entropic risk measure when the conditional mean of returns is given and the returns fall within a random closed ball. Besides, the related uniqueness problems are also solved by the different methods for the above--mentioned two conditional convex risk measures.

The remainder of this paper is organized as follows. Section \ref{section2} recapitulates some fundamentals on random functional analysis and random convex analysis. Section \ref{section3}, based on random convex analysis, further studies conditional convex risk measures and in particular conditional monotone mean--variance risk measure (including its $L^0$--Gateaux differentiability and giving a concise form of it), where we also illustrates that the monotone mean--variance risk measure and entropic risk measure are neither strictly convex nor coercive, and hence their conditional versions are neither random strictly convex nor random coercive. Finally, Section \ref{section4} first establishes a concise sufficient and necessary condition for the optimization problem of conditional monotone mean--variance to allow a solution and also considers the related uniqueness problem, and then we makes use of the theory of $L^0$--convex compactness to obtain the existence of the optimization problem for conditional entropic risk measure when the conditional mean of returns is given and the returns are within a random closed ball, at the same time the theory of $RIP$ modules is further used to obtain the uniqueness of this problem.

\section{Some fundamentals on random functional analysis and random convex analysis} \label{section2}

Throughout this paper, unless otherwise stated, $K$ denotes the scalar field $R$ of real numbers or $C$ of complex numbers, $(\Omega,\mathcal{F},P)$ a probability space and $L^0(\mathcal{F},K)$ the algebra of equivalence classes of $K$--valued $\mathcal{F}$--measurable random variables on $\Omega$, on which the scalar multiplication, addition and algebraic multiplication operations are induced from the corresponding pointwise operations of random variables. Specially, $L^0(\mathcal{F}) := L^0(\mathcal{F},R)$ and $\bar{L}^0(\mathcal{F})$ is the set of equivalence classes of extended real--valued $\mathcal{F}$--measurable random variables on $\Omega$. Here, equivalence is understood as usual, namely two random variables are equivalent if they equal $P$--almost surely. Proposition \ref{proposition2.1} below can be regarded as a random version of the classical supremum and infimum principle. The partial order $\leq$ on $\bar{L}^0(\mathcal{F})$ is defined by $\xi \leq \eta$ iff $\xi^0(\omega) \leq \eta^0(\omega)$ for almost surely all $\omega \in \Omega$, where $\xi^0$ and $\eta^0$ are arbitrarily chosen representatives of $\xi$ and $\eta$ respectively.
\begin{proposition}\label{proposition2.1} \cite{DS58}.
 $(\bar{L}^0(\mathcal{F}), \leq)$ is a complete lattice, for any nonempty subset $H$ of $\bar{L}^0(\mathcal{F})$, $\bigvee H$ and $\bigwedge H$ denote the supremum and infimum of $H$, respectively, and the following statements hold:
\begin{enumerate}[(1)]
\item There exists two sequences $\{ a_n, n \in N \}$ and $\{ b_n, n \in N \}$ in $H$ such that $\bigvee_{n \geq 1} a_n = \bigvee H$ and $\bigwedge_{n \geq 1}b_n = \bigwedge H$.
\item If $H$ is directed upwards $($downwards$)$, namely there exists $h_3 \in H$ for any $h_1$ and $h_2 \in H$ such that $h_3 \geq h_1 \bigvee h_2$ $($resp., $h_3 \leq h_1 \bigwedge h_2)$, then $\{ a_n, n \in N \}$ $($resp., $\{ b_n, n \in N \})$ can be chosen as nondecreasing $($resp., nonincreasing$)$.
\item As a sublattice of $\bar{L}^0(\mathcal{F})$, $L^0(\mathcal{F})$ is conditionally complete, namely any nonempty subset with an upper $($resp., a lower$)$ bound has a supremum $($resp., an infimum$)$.
\end{enumerate}
\end{proposition}

As usual, for $\xi$ and $\eta$ in $\bar{L}^0(\mathcal{F})$, $\xi > \eta$  means $\xi \geq \eta$ and $\xi \neq \eta$ , and for any $A \in \mathcal{F}$, $\xi > \eta$ on $A$ means $\xi^0(\omega) > \eta^0(\omega)$ for almost surely all $\omega \in A$, which $\xi^0$ and $\eta^0$ are arbitrarily chosen representatives of $\xi$ and $\eta$, respectively.

Besides, we always employ the following notations:

$L^0_+(\mathcal{F}) = \{ \xi \in L^0(\mathcal{F}) : \xi \geq 0 \}$;

$L^0_{++}(\mathcal{F}) = \{ \xi \in L^0(\mathcal{F}) : \xi > 0$ on  $ \Omega  \}$;

$I_A$ stands for the characteristic function of $A$ and $\tilde{I}_A$ the equivalence class of $I_A$ for any $A \in \mathcal{F}$.

\begin{definition}\label{definition2.2} \cite{Guo92,Guo93,YZG91}
An ordered pair $(E,\| \cdot \|)$ is called a random normed module $($briefly, an $RN$ module$)$ over the scalar field $K$ with base $(\Omega,\mathcal{F},P)$ if $E$ is a left module over the algebra $L^0(\mathcal{F},K)$ $($briefly, an $L^0(\mathcal{F},K)$--module$)$ and $\| \cdot \|$ is a mapping from $E$ to $L^0_+(\mathcal{F})$ such that the following axioms are satisfied:
\begin{enumerate}[(RNM-1)]
\item $\| \xi x \| = |\xi| \|x\|$ for any $\xi \in L^0(\mathcal{F}, K)$ and any $x \in E$;
\item $\|x+y\| \leq \|x\| + \|y\|$ for all $x$ and $y \in E$;
\item $\|x\| = 0$ implies $x = \theta$ $($the null in $E)$.
\end{enumerate}
$\| \cdot \|$ is called the $L^0$--norm on $E$ and $\|x\|$ the $L^0$--norm of $x$ for any $x \in E$. A mapping $\| \cdot \| : E \to L^0_+(\mathcal{F})$ is called an $L^0$--seminorm on $E$ if it only satisfies $(RNM-1)$ and $(RNM-2)$ as above.
\end{definition}

Similarly, one can have the notion of a random inner product module (briefly, an $RIP$ module), which was first introduced by Guo in \cite{Guo93}, see also \cite{Guo10}.

\begin{example}\label{example2.3}
$L^0(\mathcal{F},K)$ is an $RN$ module over $K$ with base $(\Omega,\mathcal{F},P)$ with its $L^0$--norm $\| \cdot \|$ defined by $\| x \| = |x|$ for any $x \in L^0(\mathcal{F},K)$, where $|x|$ denotes the absolute value of $x$, we often denote $\|\cdot\|$ by $|\cdot|$. $L^0(\mathcal{F},K)$ is also an $RIP$ module over $K$ with base $(\Omega,\mathcal{F},P)$, which is endowed with the $L^0$--inner product $(\cdot, \cdot)$ defined by $(x,y) = x \cdot \bar{y}$ for any $x$ and $y$ in $L^0(\mathcal{F},K)$, where $\bar{y}$ denotes the complex conjugate of $y$.
\end{example}

To give Example \ref{example2.4} below, which is very important in dynamic risk measures, let us first recall the notion of generalized mathematical expectation as follows. Let $(\Omega,\mathcal{E},P)$ be a probability space, $\mathcal{F}$ a $\sigma$--subalgebra of $\mathcal{E}$ and $E[\cdot~|~\mathcal{F}] : L^1(\mathcal{E}) \to L^1(\mathcal{F})$ the usual expectation operator. $E[\cdot~|~\mathcal{F}]$ can be first generalized to an operator from $\bar{L}^0_+(\mathcal{E}) := \{ x \in \bar{L}^0(\mathcal{E}) : x \geq 0 \}$ to $\bar{L}^0_+(\mathcal{F})$ by $E[x~|~\mathcal{F}] = \lim_{n \to \infty} E[x\wedge n~|~\mathcal{F}]$ for any $x \in \bar{L}^0_+(\mathcal{E})$, then $x \in L^0(\mathcal{E})$ is said to be conditionally integrable with respect to $\mathcal{F}$ (briefly, wrt $\mathcal{F}$) if $E[|x|~|~\mathcal{F}] \in L^0_+(\mathcal{F})$, at which time $E[x~|~\mathcal{F}] = E[x^+~|~\mathcal{F}] - E[x^-~|~\mathcal{F}]$, where $x^+ = x \vee 0$ and $x^- = (-x) \vee 0$. Let $L^1_{\mathcal{F}}(\mathcal{E}) = \{ x \in L^0(\mathcal{E}) : E[|x|~|~\mathcal{F}] \in L^0_+(\mathcal{F}) \}$, then it is easy to check that $x \in L^1_{\mathcal{F}}(\mathcal{E})$ if and only if (briefly, iff) $x = \xi \cdot y$ for some $\xi \in L^0(\mathcal{F})$ and $y \in L^1(\mathcal{E})$, namely $L^1_{\mathcal{F}}(\mathcal{E})$ is exactly the $L^0(\mathcal{F})$--module generated by $L^1(\mathcal{E})$. Eventually the usual expectation operator $E[\cdot~|~\mathcal{F}] : L^1(\mathcal{E}) \to L^1(\mathcal{F})$ can be extended to the current generalized expectation operator $E[\cdot~|~\mathcal{F}] : L^1_{\mathcal{F}}(\mathcal{E}) \to L^0(\mathcal{F})$. From now on, $E[\cdot~|~\mathcal{F}]$ always denotes the generalized expectation operator.
\begin{example}\cite{FKV09,HR87}\label{example2.4}
Let $p$ be a positive real number such that $1 \leq p < +\infty$, $(\Omega,\mathcal{E},P)$ and $\mathcal{F}$ as above. Let $L^p_{\mathcal{F}}(\mathcal{E})= \{ x \in L^0(\mathcal{E}) :  E[|x|^p ~|~ \mathcal{F}] \in L^0_+(\mathcal{F}) \}$, then it is easy to check that $ L^p_{\mathcal{F}}(\mathcal{E})= L^0(\mathcal{F}) \cdot L^p(\mathcal{E})$, namely $ L^p_{\mathcal{F}}(\mathcal{E})$ is the $L^0(\mathcal{F})$--module generated by $ L^p(\mathcal{E})$. Define $|||\cdot |||_p: L^p_{\mathcal{F}}(\mathcal{E}) \to L^0_{+}(\mathcal{F})$ by $|||x|||_p = E[|x|^p ~|~ \mathcal{F}]^{1/p}$ for any $x \in L^p_{\mathcal{F}}(\mathcal{E})$, then $(L^p_{\mathcal{F}}(\mathcal{E}),||| \cdot |||_p)$ is an $RN$ module over $R$ with base $(\Omega,\mathcal{F},P)$, called the conditional $L^p$--space in \cite{FKV09}. When $p=2$, $L^2_{\mathcal{F}}(\mathcal{E})$ is just the conditional Hilbert space in \cite{HR87}, which is endowed with the $L^0$--inner product $(x,y) = E[x \cdot y ~|~ \mathcal{F}]$ for any $x$ and $ y \in L^2_{\mathcal{F}}(\mathcal{E})$, then it is obviously a special $RIP$ module.
\end{example}

\begin{definition}\cite{Guo01}\label{definition2.5}
An ordered pair $(E,\mathcal{P})$ is called a random locally convex module (briefly, an $RLC$ module) over $K$ with base $(\Omega,\mathcal{F},P)$ if $E$ is an $L^0(\mathcal{F},K)$--module and $\mathcal{P}$ is a family of $L^0$--seminorms on $E$ such that $\bigvee\{ \|x\| : \| \cdot \| \in \mathcal{P} \} = 0 $ iff $x = \theta$.

Let $(E,\mathcal{P})$ be an $RLC$ module over $K$ with base $(\Omega,\mathcal{F},P)$ and $\mathcal{P}_f = \{ Q \subset \mathcal{P}: Q$ is a finite subfamily of $\mathcal{P} \}$. For any $Q \in \mathcal{P}_f$, $\|\cdot\|_Q : E \to L^0_{+}(\mathcal{F})$ defined by $\|x\|_Q = \bigvee\{ \|x\| : \| \cdot \| \in Q \}$ for any $x \in E$ is still an $L^0$--seminorm on $E$. $\mathcal{P}$ is said to have the countable concatenation property \cite{FKV09} if $\sum^{\infty}_{n=1} \tilde{I}_{A_n} \|\cdot\|_{Q_n}$ still belongs to $\mathcal{P}$ for any countable partition $\{ A_n : n \in N \}$ of $\Omega$ to $\mathcal{F}$ and any countable subset $\{ Q_n : n \in N \}$ of $\mathcal{P}_f$.
\end{definition}

\begin{proposition}\cite{Guo01}\label{proposition2.6}
Let $(E,\mathcal{P})$ be an $RLC$ module over $K$ with base $(\Omega,\mathcal{F},P)$. For any given positive numbers $\varepsilon$ and $\lambda$ such that $0 < \lambda <1$ and $Q \in \mathcal{P}_f$, let $N_{\theta}(Q,\varepsilon,\lambda) = \{ x \in E : P\{ \omega \in \Omega : \|x\|_Q(\omega) < \varepsilon \} > 1-\lambda \}$, then $\{ N_{\theta}(Q,\varepsilon,\lambda) : Q \in \mathcal{P}_f$, $\varepsilon >0, 0 < \lambda <1 \}$ forms a local base of some Hausdorff linear topology for $E$, called the $(\varepsilon,\lambda)$--topology induced by $\mathcal{P}$, denoted by $\mathcal{T}_{\varepsilon,\lambda}$. It is easy to see that the $(\varepsilon,\lambda)$--topology for $L^0(\mathcal{F},K)$ is just the usual topology of convergence in probability such that $L^0(\mathcal{F},K)$ is a topological algebra, and it is also straightforward to check that $(E,\mathcal{T}_{\varepsilon,\lambda})$ is a Hausdorff topological module over the topological algebra $L^0(\mathcal{F},K)$.
\end{proposition}

From now on, the $(\varepsilon,\lambda)$--topology for any $RLC$ module $(E,\mathcal{P})$ is always denoted by $\mathcal{T}_{\varepsilon,\lambda}$ whenever no confusion occurs. As above, $(L^0(\mathcal{F},K),\mathcal{T}_{\varepsilon,\lambda})$ is a topological algebra and $\mathcal{T}_{\varepsilon,\lambda}$ is essentially not locally convex in general, for example, the conjugate space of $(L^0(\mathcal{F},K),\mathcal{T}_{\varepsilon,\lambda})$ is trivial, namely equal to $\{ \theta \}$ when $\mathcal{F}$ is atomless. Perhaps, this leads Filipovi\'{c} et.al in \cite{FKV09} to another kind of topology for $L^0(\mathcal{F},K)$ : let $\varepsilon \in L^0_{++}(\mathcal{F})$ and $U(\varepsilon) = \{ x \in L^0(\mathcal{F},K) ~|~ |x| \leq \varepsilon \}$, then $\{ U(\varepsilon) : \varepsilon \in L^0_{++}(\mathcal{F}) \}$ forms a local base of some Hausdorff topology for $L^0(\mathcal{F},K)$ such that $L^0(\mathcal{F},K)$ becomes a topological ring. Since the topology, denoted by $\mathcal{T}_{c}$, is too strong to guarantee that $(L^0(\mathcal{F},K),\mathcal{T}_{c})$ is a linear topology (namely the scalar multiplication operation is not continuous).

Also, at the early stage, except Guo's work, all the work in \cite{FKV09,HR87,HLR91} only considers $L^0(\mathcal{F})$--modules (namely only the real case). Here, we state the work of \cite{FKV09} under the franmework of $L^0(\mathcal{F},K)$--modules for generality.
\begin{definition}\cite{GC09,FKV09}\label{definition2.7}
Let $E$ be an $L^0(\mathcal{F},K)$--module and $G$ a subset of $E$. $G$ is said to be $L^0$--convex if $\xi x + \eta y \in G$ for any $x, y \in G$ and $\xi, \eta \in L^0_+(\mathcal{F})$ such that $ \xi + \eta = 1$; $G$ is said to be $L^0$--absorbent if there exists some $\varepsilon \in L^0_{++}(\mathcal{F})$ for any $x \in E$ such that $\eta x \in G$ for any $\eta \in L^0(\mathcal{F}, K)$ such that $|\eta| \leq \varepsilon$; $G$ is said to be $L^0$--balanced if $\eta G \subset G$ for any $\eta \in L^0(\mathcal{F}, K)$ such that $|\eta| \leq 1$. Finally, a subset $F$ of $E$ is said to be absorbed by $G$ if there exists $\varepsilon \in L^0_{++}(\mathcal{F})$ such that $\eta F \subset G$ for any $\eta \in L^0(\mathcal{F},K)$ such that $|\eta| \leq \varepsilon$.
\end{definition}

A topological space $(E,\mathcal{T})$ is called a topological $L^0$--module over $L^0(\mathcal{F},K)$ \cite{FKV09} if $E$ is an $L^0(\mathcal{F},K)$--module and $(E,\mathcal{T})$ is a topological module over the topological ring  $(L^0(\mathcal{F},K),\mathcal{T}_{c})$. Filipovi\'{c} et.al \cite{FKV09} further called a topological $L^0$--module $(E,\mathcal{T})$ a locally $L^0$--convex module if $\mathcal{T}$ has a local base consisting of $L^0$--convex, $L^0$--absorbent and $L^0$--balanced subsets of $E$, at which time $\mathcal{T}$ is called a local $L^0$--convex topology for $E$. This leads to the following:
\begin{proposition}\cite{FKV09}\label{proposition2.8}
Let $(E,\mathcal{P})$ be an $RLC$ module over $K$ with base $(\Omega,\mathcal{F},P)$. Given $Q \in \mathcal{P}_f$ and $\varepsilon \in L^0_{++}(\mathcal{F})$, let $U(Q,\varepsilon) = \{ x \in E : \|x\|_Q \leq \varepsilon \}$, then $\{ U(Q,\varepsilon) : Q \in \mathcal{P}_f$ and $\varepsilon \in L^0_{++}(\mathcal{F}) \}$ forms a local base of some Hausdorff locally $L^0$--convex topology, called the locally $L^0$--convex topology $($for $E)$ induced by $\mathcal{P}$.
\end{proposition}

For an $RLC$ module $(E,\mathcal{P})$, from now on we always denote its locally $L^0$--convex topology by $\mathcal{T}_c$ when no confusion occurs. A key bridge linking $\mathcal{T}_{\varepsilon,\lambda}$ and $\mathcal{T}_{c}$ is the following notion--$\sigma$--stability, also please bear in mind that $\mathcal{T}_{c}$ is much stronger than $\mathcal{T}_{\varepsilon,\lambda}$ in general.
\begin{definition}\cite{Guo10}\label{definition2.9}
Let $E$ be an $L^0(\mathcal{F},K)$--module and $G$ a subset of $E$. $G$ is said to be $\sigma$--stable (or, to have the countable concatenation property in the original terminology of \cite{Guo10} $)$ if there exists $x \in G$ for any countable partition $\{ A_n : n \in N \}$ of $\Omega$ to $\mathcal{F}$ and any sequence $\{ x_n : n \in N \}$ of $G$ such that $\tilde{I}_{A_n} \cdot x = \tilde{I}_{A_n} \cdot x_n$ for each $n \in N$. If $E$ itself is $\sigma$--stable and $G$ is any subset of $E$, the smallest $\sigma$--stable set including $G$, denoted by $H_{cc}(G)$, is called the $\sigma$--stable hull of $G$.
\end{definition}

As pointed out in \cite{Guo10}, when $(E,\mathcal{P})$ is an $RLC$ module, $x$ in Def. \ref{definition2.9} is unique, denoted by $\sum^{\infty}_{n=1}\tilde{I}_{A_n} \cdot x_n$. A slight generalization of Def. \ref{definition2.9} is as follows: $G$ is said to be relatively $\sigma$--stable if there exists $x \in E$ for any countable partition $\{ A_n : n \in N \}$ of $\Omega$ to $\mathcal{F}$ and any countable subset $\{ x_n : n \in N \}$ of $G$ such that $\tilde{I}_{A_n} \cdot x = \tilde{I}_{A_n} \cdot x_n$ for each $n \in N$, then $x \in G$. The notion of relative $\sigma$--stability was independently introduced in \cite{WG15,Zap17}, by which Theorem 2.10 below was independently proved in \cite{WG15,YZG91}
\begin{theorem}\label{theorem2.10}\cite{WG15,YZG91}
For a Hausdorff locally $L^0$--convex module $(E,\mathcal{T})$ over $L^0(\mathcal{F},K)$, the locally $L^0$--convex topology $\mathcal{T}$ can be induced by a family $\mathcal{P}$ of $L^0$--seminorms on $E$ iff $\mathcal{T}$ has a local base consisting of  $L^0$--convex, $L^0$--absorbent and $L^0$--balanced subsets which are relatively $\sigma$--stable.
\end{theorem}

\begin{remark}\label{remark2.11}
Theorem 2.10 corrects Theorem 2.4 of \cite{FKV09}.
\end{remark}
\par
Theorem \ref{theorem2.12} below also plays a crucial role in linking $\mathcal{T}_{\varepsilon,\lambda}$ and $\mathcal{T}_c$.

\begin{theorem}\label{theorem2.12}\cite{Guo10}
Let $G$ be a subset of an $RLC$ module $(E,\mathcal{P})$ such that $G$ is $\sigma$--stable, then $\bar{G}_{\varepsilon,\lambda}=\bar{G}_c$, where $\bar{G}_{\varepsilon,\lambda}$ and $\bar{G}_c$ stand for the closures of $G$ under $\mathcal{T}_{\varepsilon,\lambda}$ and $\mathcal{T}_c$, respectively.
\end{theorem}

\begin{definition}\label{definition2.13}\cite{Guo10}
Let $(E,\mathcal{P})$ be an $RLC$ module over $K$ with base $(\Omega,\mathcal{F},P)$. $E^*_{\varepsilon,\lambda}=\{f:E\rightarrow L^0(\mathcal{F},K)~|~f$ is a continuous module homomorphism from $(E,\mathcal{T}_{\varepsilon,\lambda})$ to $(L^0(\mathcal{F},K), \mathcal{T}_{\varepsilon,\lambda})\}$ is called the random conjugate space of $E$ under the $(\varepsilon,\lambda)$--topology $\mathcal{T}_{\varepsilon,\lambda}$; $E^*_c=\{f:E\rightarrow L^0(\mathcal{F},K)~|~f$ is a continuous module homomorphism from $(E,\mathcal{T}_c)$ to $(L^0(\mathcal{F},K), \mathcal{T}_c)\}$ is called the random conjugate space of $E$ under the  locally $L^0$--convex topology $\mathcal{T}_c$.
\end{definition}

\begin{theorem}\label{theorem2.14}\cite{Guo10,GZZ15a}
Let $(E,\mathcal{P})$ be an $RLC$ module over $K$ with base $(\Omega,\mathcal{F},P)$. Then $E^*_{\varepsilon,\lambda}=E^*_c$ if $\mathcal{P}$ has the countable concatenation property, specially $E^*_{\varepsilon,\lambda}=E^*_c$ for any $RN$ module $(E,\|\cdot\|)$, see \cite{Guo10}. Generally, $E^*_{\varepsilon,\lambda}=H_{cc}(E^*_c)$, see \cite{GZZ15a}.
\end{theorem}

\par
By Theorem\ref{theorem2.14}, $E^*_{\varepsilon,\lambda}=E^*_c$ for any $RN$ module, and thus we can use $E^*$ for $E^*_{\varepsilon,\lambda}$ or $E^*_c$ for an $RN$ module. Also, let $(E,\|\cdot\|)$ be an $RN$ module over $K$ with base $(\Omega,\mathcal{F},P)$, it is earlier proved by Guo in \cite{Guo93} that a linear operator $f: E\rightarrow L^0(\mathcal{F},K)$ is almost surely $(a.s.)$ bounded iff $f\in E^*$, where $f$ is a.s. bounded if there exists $\xi\in L^0_+(\mathcal{F})$ such that $|f(x)|\leq \xi \|x\|$ for any $x\in E$. Further, $E^*$ becomes an $RN$ module over $K$ with base $(\Omega,\mathcal{F},P)$ when $E^*$ is endowed with the $L^0$--norm $\|\cdot\|^*: E^*\rightarrow L^0_+(\mathcal{F})$ defined by $\|f\|^*=\bigwedge\{\xi\in L^0_+(\mathcal{F})~|~ |f(x)|\leq \xi\|x\|$ for any $x\in E\}$ for any $f\in E^*$. In fact, $\|f\|^*=\bigvee\{|f(x)|: x\in E$ and $\|x\|\leq 1\}$ for any $f\in E^*$, it also holds that $(E^*,\|\cdot\|^*)$ is $\mathcal{T}_{\varepsilon,\lambda}$--complete. As usual $(E^*,\|\cdot\|^*)$ is still denoted by $(E^*,\|\cdot\|)$.

\begin{theorem}\label{theorem2.15}\cite{Guo10}
Let $(E,\mathcal{P})$ be an $RLC$ module. Then $E$ is $\mathcal{T}_{\varepsilon,\lambda}$--complete iff both $E$ is $\mathcal{T}_c$--complete and $E$ is $\sigma$--stable.
\end{theorem}

\par
To introduce the following separation theorem--Theorem\ref{theorem2.16}, which is a refined formulation of Theorem3.1 of \cite{GXC09}, let $(E,\mathcal{P})$ be an $RLC$ module over $K$ with base $(\Omega, \mathcal{F},P)$, $x\in E$ and $G$ a nonempty subset of $E$, for each finite subfamily $Q$ of $\mathcal{P}($ namely $Q\in \mathcal{P}_f)$, define $d_Q(x,G)=\bigwedge\{\|x-y\|_Q: y\in G\}$, then $d(x,G):=\bigvee\{d_Q(x,G): Q\in \mathcal{P}_f\}$ is called the random distance from $x$ to $G$. It is easy to prove that $x\in \bar{G}_{\varepsilon,\lambda}$ iff $d(x,G)=0$.

\begin{theorem}\label{theorem2.16}\cite[Proposition4.1]{GZZ15a}
Let $(E,\mathcal{P})$ be an $RLC$ module over $K$ with base $(\Omega,\mathcal{F},P)$, $x\in E$ and $G$ a nonempty $\mathcal{T}_{\varepsilon,\lambda}$--closed $L^0$--convex subset such that $x\notin G$. Then there exists $f\in E^*_{\varepsilon,\lambda}$ satisfying the following two items:
\begin{enumerate}[(1)]
\item $Re(f(x))>\bigvee \{Re(f(y)): y\in G\}$ on $(d(x,G)>0)$;
\item $Re(f(x))=\bigvee \{Re(f(y)): y\in G\}$ on $(d(x,G)>0)^c$.
\end{enumerate}
Here $Re(f(x))$ stands for the real part of $f(x)$.
\end{theorem}

\begin{corollary}\label{corollary2.17}
Let $(E,\mathcal{P})$ be an $RLC$ module over $K$ with base $(\Omega,\mathcal{F},P)$ such that $\mathcal{P}$ has the countable concatenation property, $x\in E$ and $G$ a nonempty $\mathcal{T}_c$--closed $L^0$--convex and $\sigma$--stable subset of $E$ such that $x\notin G$. Then there exists $f\in E^*_c$ satisfying $(1)$ and $(2)$ of Theorem\ref{theorem2.16}.
\end{corollary}

\begin{definition}\label{definition2.18}
Let $(E,\mathcal{P})$ be an $RLC$ module over $R$ with base $(\Omega,\mathcal{F},P)$. A function $f: E\rightarrow \bar{L}^0(\mathcal{F})$ is $L^0$--convex if $f(\xi x+ (1-\xi)y)\leq \xi f(x)+ (1-\xi)f(y)$ for all $x$ and $y\in E$ and $\xi\in L^0_+(\mathcal{F})$ with $0\leq \xi \leq 1$, where we make the following convention: $0\cdot(\pm\infty)=0$ and $(-\infty)+(+\infty)=+\infty$; $f$ is proper if $f(x)>-\infty$ on $\Omega$ for any $x\in E$ and $Dom(f):=\{x\in E~|~f(x)<+\infty$ on $\Omega\}\neq\emptyset$; $f$ is local if $\tilde{I}_Af(x)=\tilde{I}_Af(\tilde{I}_Ax)$ for any $x\in E$ and $A\in \mathcal{F}$; a proper $f$ is $\mathcal{T}_{\varepsilon,\lambda}$--lower semicontinuous if $epi(f):=\{(x,r)\in E\times L^0(\mathcal{F})~|~ f(x)\leq r\}$ is closed in $(E,\mathcal{T}_{\varepsilon,\lambda})\times (L^0(\mathcal{F}), \mathcal{T}_{\varepsilon,\lambda})$; a proper $f$ is $\mathcal{T}_c$--lower semicontinuous if $\{x\in E~|~f(x)\leq r\}$ is $\mathcal{T}_c$--closed for any $r\in L^0(\mathcal{F})$.
\end{definition}

\begin{remark}\label{remark2.19}
Def.\ref{definition2.18} is enough for this paper, see \cite[Def.2.22]{GZWYYZ17} for the notion of a closed function, which is more general than that of a proper lower semicontinuous $L^0$--convex function, see \cite{GZWW20} for an $L^0$--convex function only defined on an $L^0$--convex subset.
\end{remark}

\begin{theorem}\label{theorem2.20}\cite[Theorem2.13]{GZWYYZ17}
Let $(E,\mathcal{P})$ be an $RLC$ module over $R$ with base $(\Omega,\mathcal{F},P)$ such that both $E$ is $\sigma$--stable and $\mathcal{P}$ has the countable concatenation property, $f:E \rightarrow \bar{L}^0(\mathcal{F})$ a proper and local function, then the following are equivalent:
\begin{enumerate}[(1)]
\item $\{x\in E~|~ f(x)\leq r\}$ is $\mathcal{T}_{\varepsilon,\lambda}$--closed for any $r\in L^0(\mathcal{F})$;
\item $f$ is $\mathcal{T}_c$--lower semicontinuous;
\item $f$ is $\mathcal{T}_{\varepsilon,\lambda}$--lower semicontinuous;
\item $epi(f)$ is closed in $(E,\mathcal{T}_c)\times (L^0(\mathcal{F}), \mathcal{T}_c)$;
\item $\underline{\lim}_{\alpha}f(x_{\alpha})\geq f(x)$ for any $x\in E$ and any net $\{x_{\alpha}, \alpha\in \Gamma\}$ in $E$ convergent to $x$ with respect to $\mathcal{T}_c,$ where $\underline{\lim}_{\alpha}f(x_{\alpha})= \bigvee_{\beta\in \Gamma}(\bigwedge_{\alpha\geq \beta} f(x_{\alpha}))$.
\end{enumerate}
\end{theorem}

\begin{remark}\label{remark2.21}
The equivalence among $(2),(4)$ and $(5)$ in Theorem\ref{theorem2.20} was first studied in \cite{FKV09}, but the proof given in \cite{FKV09} is not only not strict but also depends on a wrong result--Lemma2.28 of \cite{FKV09}. A strict and complete proof of Theorem\ref{theorem2.20} was first given in \cite{GZWYYZ17}, and hence the strict and complete proof of Corollary\ref{corollary2.23} below was first given in \cite{GZWYYZ17}.
\end{remark}

\par
Let $(E,\mathcal{P})$ be an $RLC$ module over $R$ with base $(\Omega,\mathcal{F},P)$ and $f:E \rightarrow \bar{L}^0(\mathcal{F})$ a proper $L^0$--convex function, random conjugate and double random conjugate functions under $\mathcal{T}_{\varepsilon,\lambda}$ and $\mathcal{T}_c$ can be defined as follows:\\
$f^*_{\varepsilon,\lambda}: E^*_{\varepsilon,\lambda} \rightarrow \bar{L}^0(\mathcal{F})$ is defined by $f^*_{\varepsilon,\lambda}(g)=\bigvee\{g(x)-f(x): x\in E\}$ for any $g\in E^*_{\varepsilon,\lambda}$;\\
$f^{**}_{\varepsilon,\lambda}: E\rightarrow \bar{L}^0(\mathcal{F})$ is defined by $f^{**}_{\varepsilon,\lambda}(x)=\bigvee\{g(x)-f^*_{\varepsilon,\lambda}(g): g\in E^*_{\varepsilon,\lambda} \}$ for any $x\in E$;\\
$f^*_c: E^*_c \rightarrow \bar{L}^0(\mathcal{F})$ is defined by $f^*_c(g)=\bigvee\{g(x)-f(x): x\in E\}$ for any $g\in E^*_c$;\\
$f^{**}_c: E\rightarrow \bar{L}^0(\mathcal{F})$ is defined by $f^{**}_c(x)=\bigvee\{g(x)-f^*_c(g): g\in E^*_c \}$ for any $x\in E$.

\begin{theorem}\label{theorem2.22}\cite[Theorem5.1]{GZZ15a}
Let $(E,\mathcal{P})$ be an $RLC$ module over $R$ with base $(\Omega,\mathcal{F},P)$ and $f:E \rightarrow \bar{L}^0(\mathcal{F})$ a proper $\mathcal{T}_{\varepsilon,\lambda}$--lower semicontinuous $L^0$--convex function. Then $f^{**}_{\varepsilon,\lambda}=f$.
\end{theorem}

\begin{corollary}\label{corollary2.23}\cite{GZWYYZ17}
Let $(E,\mathcal{P})$ be an $RLC$ module over $R$ with base $(\Omega,\mathcal{F},P)$ such that $E$ is $\sigma$--stable and $f:E \rightarrow \bar{L}^0(\mathcal{F})$ a proper $\mathcal{T}_c$--lower semicontinuous $L^0$--convex function. Then $f^{**}_c=f$.
\end{corollary}

\begin{remark}\label{remark2.24}
Corollary\ref{corollary2.23} still holds for a $\mathcal{T}_c$--closed function, see \cite[Proposition2.25]{GZWYYZ17}. Corollary\ref{corollary2.23} corrects Theorem3.8 of \cite{FKV09}, where a locally $L^0$--convex module was employed, Theorem\ref{theorem2.10} shows that a locally $L^0$--convex  module is too large for Theorem3.8 of \cite{FKV09}. Besides, the countable concatenation property employed in \cite{FKV09} amounts to the fact that $\mathcal{P}$ has this property, rather than the condition that $E$ is $\sigma$--stable, which is essential for the proof of Corollary\ref{corollary2.23}.
\end{remark}

\par
The theory of random duality was first studied in \cite{GC09} under the $(\varepsilon,\lambda)$--topology and further in \cite{GZZ15b} under the locally $L^0$--convex topology.

\begin{definition}\label{definition2.25}\cite{GC09}
Let $X$ and $Y$ be two $L^0(\mathcal{F},K)$--modules and $\langle\cdot, \cdot\rangle: X\times Y\rightarrow L^0(\mathcal{F},K)$ an $L^0$--bilinear function $($namely $\langle x,\cdot\rangle: Y\rightarrow L^0(\mathcal{F},K)$ and $\langle\cdot,y\rangle: X\rightarrow L^0(\mathcal{F},K)$ are both module homomorphisms for any fixed $x$ in $X$ and $y$ in $Y)$. $\langle X,Y\rangle$ is called a random duality $($or a random dual pair$)$ over $K$ with base $(\Omega,\mathcal{F},P)$ if the following are satisfied:
\begin{enumerate}[(1)]
\item $\langle x,y\rangle=0$ for any $y\in Y$ implies $x=\theta$;
\item $\langle x,y\rangle=0$ for any $x\in X$ implies $y=\theta$.
\end{enumerate}
\end{definition}

\par
The study of random consistent topology and random admissible topology was carried out in \cite{GZZ15b}, in particular the random bipolar theorem was obtained in \cite{GZZ15b}. Let $\langle X,Y\rangle$ be a random duality over $K$ with base $(\Omega,\mathcal{F},P)$, further let $\mathcal{B}(Y,X)$ be the family of subsets $A$ of $Y$ such that $\bigvee\{|\langle x, y\rangle|: y\in A\}\in L^0_+(\mathcal{F})$ for any $x\in X$. For each $A\in \mathcal{B}(Y,X)$, define the $L^0$--seminorm $\|\cdot\|_A: X\rightarrow L^0_+(\mathcal{F})$ by $\|x\|_A=\bigvee\{|\langle x, y\rangle|: y\in A\}$, then it is easy to see that $(X, \{\|\cdot\|_A\}_{A\in \mathcal{B}(Y,X)})$ is an $RLC$ module over $K$ with base $(\Omega,\mathcal{F},P)$, the locally $L^0$--convex topology on $X$ induced by $\{\|\cdot\|_A\}_{A\in \mathcal{B}(Y,X)})$ is denoted by $\beta_c(X,Y)$.

\begin{definition}\label{definition2.26}\cite{FKV09,GZZ15b}
Let $(E,\mathcal{T})$ be a locally $L^0$--convex module over $L^0(\mathcal{F},K)$. An $L^0$--balanced, $L^0$--absorbent closed $L^0$--convex subset of $E$ is called an $L^0$--barrel, further $(E,\mathcal{T})$ is called an $L^0$--barreled module\cite{FKV09} if every $L^0$--barrel is a neighborhood of $\theta$. $(E,\mathcal{T})$ is called an $L^0$--pre--barreled module\cite{GZZ15b} if every $\sigma$--stable barrel is a neighborhood of $\theta$.
\end{definition}

\par
Theorem\ref{theorem2.27} below shows that only the notion of an $L^0$--pre--barreled module is useful in random convex analysis and its applications to conditional risk measures. Clearly, $\langle E, E^*_c\rangle$ forms a random duality for any $RLC$ module $(E,\mathcal{P})$.

\begin{theorem}\label{theorem2.27}\cite{GZZ15b}
Let $(E,\mathcal{P})$ be an $RLC$ module such that $E$ is $\sigma$--stable. Then $(E,\mathcal{T}_c)$ is $L^0$--pre--barreled iff $\mathcal{T}_c=\beta_c(E,E^*_c)$.
\end{theorem}

\begin{corollary}\label{corollary2.28}\cite{GZZ15b}
Let $(E,\|\cdot\|)$ be a $\mathcal{T}_c$--complete $RN$ module such that $E$ is $\sigma$--stable, then $(E,\mathcal{T}_c)$ is $L^0$--pre--barreled.
\end{corollary}

\par
Since it was in \cite{GZWYYZ17} that we gave a strict proof of Theorem \ref{theorem2.20}, we were forced to employ a stronger notion of a $\mathcal{T}_c$--lower semicontinuous $L^0$--convex function $f$ (namely $epi(f)$ is $\mathcal{T}_c$--closed) in the proofs of Theorems 4.2 and 4.3 of \cite{GZZ15b}. Combining Theorem\ref{theorem2.20} and \cite[Theorem4.2]{GZZ15b}, we can now have:

\begin{theorem}\label{theorem2.29}
Let $(E,\mathcal{P})$ be an $L^0$--pre--barreled $RLC$ module over $R$ with base $(\Omega,\mathcal{F},P)$ such that $E$ is $\sigma$--stable and $\mathcal{P}$ has the countable concatenation property. Then a proper $\mathcal{T}_c$--lower semicontinuous $L^0$--convex function $f: E \rightarrow \bar{L}^0(\mathcal{F})$ is $\mathcal{T}_c$--continuous on $int(Dom(f))$, namely $f: (int(Dom(f)), \mathcal{T}_c)\rightarrow (L^0(\mathcal{F}), \mathcal{T}_c)$ is continuous, where $int(Dom(f))$ denotes the $\mathcal{T}_c$--interior of $Dom(f)$.
\end{theorem}

\par
Let $(E,\mathcal{P})$ be an $RLC$ module over $R$ with base $(\Omega,\mathcal{F},P)$ and $f: E \rightarrow \bar{L}^0(\mathcal{F})$ a proper $L^0$--convex function. $g\in E^*_c$ is called an $L^0$--subgradient of $f$ at $x\in Dom(f)$ if $g(y-x)\leq f(y)-f(x)$ for any $y\in E$, denote by $\partial f(x)$ the set of $L^0$--subgradients of $f$ at $x$, called the $L^0$--subdifferential of $f$ at $x$.
\par
Combing Theorem\ref{theorem2.20} and \cite[Theorem4.3]{GZZ15b}, we can now have:
\begin{theorem}\label{theorem2.30}
Let $(E,\mathcal{P})$ be an $L^0$--pre--barreled $RLC$ module over $R$ with base $(\Omega,\mathcal{F},P)$ such that $E$ is $\sigma$--stable and $\mathcal{P}$ has the countable concatenation property. Then every proper $\mathcal{T}_c$--lower semicontinuous $L^0$--convex function $f: E \rightarrow \bar{L}^0(\mathcal{F})$ is $L^0$--subdifferentiable on $int(Dom(f))$, namely $\partial f(x)\neq \emptyset$ for any $x\in int(Dom(f))$.
\end{theorem}

\par
Let $(E,\mathcal{P})$ be an $RLC$ module over $R$ with base $(\Omega,\mathcal{F},P)$ and  $f: E \rightarrow \bar{L}^0(\mathcal{F})$ a proper $L^0$--convex function, $f$ is said to be $L^0$--G\^{a}teaux--differentiable at $x\in Dom(f)$  if there exists $g\in E^*_c$ such that $g(h)=\lim_{t\downarrow0}\frac{f(x+th)-f(x)}{t}$ for all $h\in E$, $g$ is denoted by $f'(x)$.

\begin{theorem}\label{theorem2.31}\cite[Theorem5.7]{GZWYYZ17}
Let $(E,\mathcal{P})$ be an $RLC$ module over $R$ with base $(\Omega,\mathcal{F},P)$ and  $f: E \rightarrow \bar{L}^0(\mathcal{F})$ a proper $L^0$--convex function. If $f$ is $L^0$--G\^{a}teaux--differentiable at $x_0\in E$, then it is $L^0$--subdifferentiable at $x_0$ and $\partial f(x_0)=\{f'(x_0)\}$. Conversely, if $f$ is $\mathcal{T}_c$--continuous at $x_0$ and has only one $L^0$--subgradient, then $f$ is $L^0$--G\^{a}teaux--differentiable at $x_0$ and $\partial f(x_0)=\{f'(x_0)\}$.
\end{theorem}

\begin{remark}\label{remark2.31}
All the results of $L^0$--convex sets and $L^0$--convex functions listed in this paper improve or correct those that were originally stated in \cite{FKV09} under the framework of a locally $L^0$--convex module. It should be also pointed out that Theorem 2.6 of \cite{FKV09} is one exception, which is both correct and very powerful tool for separating an $L^0$--convex set from an open $L^0$--convex set. Adding the theorem to the results listed in this paper will lead to a perfect random convex analysis, which will provide a solid analytic foundation for applications of random convex analysis to dynamic mathematical finance!
\end{remark}

\section{On conditional convex risk measures}\label{section3}
\begin{proposition}\cite{FKV09,Guo10}\label{proposition3.1}
Let $p$ and $q$ be a pair of H\"{o}lder conjugate numbers such that $1 \leq p < +\infty$. Then $T: (L^q_{\mathcal{F}}(\mathcal{E}), |||\cdot|||_q) \to (L^p_{\mathcal{F}}(\mathcal{E}), |||\cdot|||_p)^*$, defined by $T_y (x) = E[x\cdot y ~|~\mathcal{F}]$ for any $x \in L^p_{\mathcal{F}}(\mathcal{E})$ and $y \in L^q_{\mathcal{F}}(\mathcal{E})$, is an isometrically isomorphism in the sense of an $RN$ module, where $T_y$ denotes $T(y)$ for any $y \in L^q_{\mathcal{F}}(\mathcal{E})$ and $|||y|||_\infty = \bigwedge \{ \xi \in L^0_+(\mathcal{F}) ~|~ |y| \leq \xi \}$ for any $y \in L^{\infty}_{\mathcal{F}}(\mathcal{E})$.
\end{proposition}

\begin{remark}
In Proposition \ref{proposition3.1}, the surjectivity of $T$ was proved in \cite{FKV09}, and the isometry of $T$ was proved in \cite{Guo10}. Further, Guo proved in \cite{GZZ14} that $L^p_{\mathcal{F}}(\mathcal{E}) = H_{cc}(L^p(\mathcal{E}))$ and earlier pointed out in \cite{Guo10} that $L^p_{\mathcal{F}}(\mathcal{E})$ is $\sigma$--stable and $\mathcal{T}_{\varepsilon,\lambda}$--complete.
\end{remark}

From now on, we always assume that $1 < p < +\infty$.
\begin{definition}\label{definition3.3}\cite{FKV09,FKV12}
A proper $\mathcal{T}_{c}$--lower semicontinuous $L^0$--convex function $f: L^p_{\mathcal{F}}(\mathcal{E}) \to \bar{L}^0(\mathcal{F})$ is called a conditional convex risk measure if the following are also satisfied:
\begin{enumerate}[(1)]
\item $f(x) \leq f(y)$ whenever $x \geq y$;
\item $f(x+m) = f(x) - m$ for any $x \in L^p_{\mathcal{F}}(\mathcal{E})$ and $m \in L^0(\mathcal{F})$.
\end{enumerate}
\end{definition}

Let $\mathcal{P}_c = \{ x \in L^p_{\mathcal{F}}(\mathcal{E}) ~|~ x \geq 0 \}$, $\mathcal{P}_c^0 = \{ y \in L^q_{\mathcal{F}}(\mathcal{E}) ~|~ y \leq 0 \}$ and $\mathcal{D} = \{ y \in L^q_{\mathcal{F}}(\mathcal{E}) ~|~ E[y ~|~\mathcal{F}] = -1\}$. Then Proposition \ref{proposition3.4} below was proved in \cite{FKV12}.
\begin{proposition}\cite{FKV12}\label{proposition3.4}
A proper $\mathcal{T}_{c}$--lower semicontinuous $L^0$--convex function $f: L^p_{\mathcal{F}}(\mathcal{E}) \to \bar{L}^0(\mathcal{F})$ satisfies (1) of Def. \ref{definition3.3} iff $Dom(f^*) \subset \mathcal{P}_c^0$. $f$ satisfies (2) of Def. \ref{definition3.3} iff $Dom(f^*) \subset \mathcal{D}$. Finally, $f$ satisfies (1) and (2) of Def. \ref{definition3.3} iff $Dom(f^*) \subset \mathcal{P}_c^0  \cap \mathcal{D}$, at which time $f(x) = \bigvee \{ E[xy~|~\mathcal{F}] - f^*(y) ~|~ y \in \mathcal{P}_c^0 \cap  \mathcal{D} \}$ for any $x \in L^p_{\mathcal{F}}(\mathcal{E})$.
\end{proposition}

\begin{example}\label{example3.5}
$\rho_\gamma : L^p_{\mathcal{F}}(\mathcal{E}) \to \bar{L}^0(\mathcal{F})$, defined by $\rho_\gamma (x) = \frac{1}{\gamma} \ln E[e^{-\gamma x}~|~\mathcal{F}]$ for any $x \in L^p_{\mathcal{F}}(\mathcal{E})$, is a conditional convex risk measure, called conditional entropic risk measure with a risk aversion index  $\gamma >0$.
\end{example}

\begin{example}\cite{MMRT09,FKV12}\label{example3.6}
Let $U_{\beta} : L^2_{\mathcal{F}}(\mathcal{E}) \to L^0(\mathcal{F})$ be defined by $U_\beta(x) = - E[x~|~\mathcal{F}] + \frac{\beta}{2}D[x~|~\mathcal{F}]$ for any $x \in L^2_{\mathcal{F}}(\mathcal{E})$, called the conditional mean--variance risk measure with a risk aversion index $\beta >0$, where $D[x~|~\mathcal{F}]$ denotes the conditional variance of $x$ with respect to $\mathcal{F}$, namely $D[x~|~\mathcal{F}] = E[(x - E[x~|~\mathcal{F}])^2~|~\mathcal{F}] = E[x^2~|~\mathcal{F}] - (E[x~|~\mathcal{F}])^2$. It is clear that $U_\beta$ is a $\mathcal{T}_{c}$--continuous $L^0$--convex function satisfying (2) but not (1) of Def. \ref{definition3.3}, and hence not a conditional convex risk measure. $V_{\beta} : L^2_{\mathcal{F}}(\mathcal{E}) \to L^0(\mathcal{F})$, which is defined by $V_\beta(x) = \bigvee \{ E[xy~|~\mathcal{F}]- U^*_\beta(y)~|~y \in \mathcal{P}_c^0 \cap  \mathcal{D} \}$ for any $x \in L^2_{\mathcal{F}}(\mathcal{E})$, is a conditional convex risk measure, called the monotone conditional mean--variance risk measure.
\end{example}

\begin{remark}\label{remark3.7}
Filipovi\'{c} et.al proved in \cite{FKV12} that $U_\beta$ is $L^0$--subdifferentiable on $L^2_{\mathcal{F}}(\mathcal{E})$, $Dom(U^*_\beta) = \mathcal{D}$ and $U^*_\beta(y) = \frac{1}{2\beta} E[(1+y)^2~|~\mathcal{F}]$ for any $y \in \mathcal{D}$. In fact, it is obvious that $U_\beta$ is $L^0$--Gateaux differentiable and $U'_\beta(x) = \beta (x - E[x~|~\mathcal{F}]) -1$ for any $x \in L^2_{\mathcal{F}}(\mathcal{E})$. $G_\beta = \{ x \in L^2_{\mathcal{F}}(\mathcal{E})~|~U'_\beta(x)  \leq 0 \} = \{  x \in L^2_{\mathcal{F}}(\mathcal{E})~|~x-E[x~|~\mathcal{F}] \leq \frac{1}{\beta} \}$ is called the domain of monotonicity of $U_\beta$. When $\mathcal{F}$ is trivial, namely $\mathcal{F} = \{ \Omega, \emptyset \}$, $V_\beta$ was introduced and thoroughly studied in \cite{MMRT09}, for a general $\mathcal{F}, V_\beta$ was first introduced and studied in \cite{FKV12}.
\end{remark}

To further study conditional convex risk measures, it is convenient to introduce the following terminologies:
\begin{definition}\cite{GZWW20,GZWYYZ17}
Let $(E,\mathcal{P})$ be an $RLC$ module over $R$ with base $(\Omega,\mathcal{F},P)$ and $V$ a nonempty $L^0$--convex subset of $E$. A proper $L^0$--convex function $f: V\to \bar{L}^0(\mathcal{F})$ is said to be random strictly convex (or, strictly $L^0$--convex) if $f(\xi x+(1-\xi)y)<\xi f(x)+(1-\xi)f(y)$ on $(0<\xi<1)\cap(x\neq y)$ for any $x,y\in V$ and $\xi \in L^0_+(\mathcal{F})$ with $0\leq\xi\leq 1$, where $(x\neq y)=(\bigvee\{\|x-y\| : \|\cdot\|\in\mathcal{P}\}>0)$. Here, we make the convention: for any $\xi$ and $\eta$ in $\bar{L}^0(\mathcal{F})$, arbitarily choosing their respective representatives $\xi^0$ and $\eta^0$, the set $\{ \omega \in \Omega : \xi^0(\omega) > \eta^0(\omega) \}$ is denoted by $(\xi > \eta)$, which would not produce any confusion, since $\{ \omega \in \Omega : \xi^0(\omega) > \eta^0(\omega) \}$ only differs by a null set for different choices of $\xi^0$ and $\eta^0$. If $(E,\mathcal{P})$ is an $RN$ module $(E,\|\cdot\|)$, a proper $L^0$--convex function $f : V\to \bar{L}^0(\mathcal{F})$ is said to be random coercive on $V$ if $\{f(x_n)~|~ n\in N\}$ converges a.s. to $+\infty$ on $A$ for any sequence $\{x_n ~|~ n\in N\}$ in $V$ and any $A\in \mathcal{F}$ with $P(A)>0$ such that $\{\|x_n\| ~|~ n\in N\}$ converges a.s. to $+\infty$ on $A$.
\end{definition}

Similarly to the discussions of static risk measures \cite[Section 5]{CL08}, for any conditional convex risk measure $f: L^p_{\mathcal{F}}(\mathcal{E}) \to \bar{L}^0(\mathcal{F})$, $\xi \in L^0_{+}(\mathcal{F})$ with $0 \leq \xi \leq 1$ and any $x,y \in L^p_{\mathcal{F}}(\mathcal{E})$ such $x-y = m \in L^0(\mathcal{F})$, it always holds that $f(\xi x+(1-\xi)y)=\xi f(x)+(1-\xi)f(y)$. Thus all conditional convex risk measures are not strictly $L^0$--convex. In the final part of this section, we illustrate that neither of the two conditional convex risk measures--monotone conditional mean--variance and conditional entropic risk measure is random coercive on the standard Markowitz--type constraint set. Now, this section will be focused on the study of $V_\beta$ as follows.

\begin{lemma}\label{lemma3.9}
For any given $x \in L^2_{\mathcal{F}}(\mathcal{E})$ and $z \in L^0_{++}(\mathcal{F}) \cap L^1(\mathcal{F})$, there exists unique one $k_x \in L^0(\mathcal{F})$ such that $E[(k_x-x)^+~|~\mathcal{F}] = z$.
\end{lemma}
\begin{proof}
Let $k_x = \bigwedge \{ \bar{y} \in L^0(\mathcal{F})~|~E[(\bar{y}-x)^+~|~\mathcal{F}] \geq z \}$, then it follows from \cite[Lemma 5.6]{FKV12} that $E[(k_x-x)^+~|~\mathcal{F}] = z$, we will prove the uniqueness of solutions $\hat{y}$ of the equation $E[(\hat{y}-x)^+~|~\mathcal{F}] = z$ as follows.

Assume that $\tilde{y}$ is another solution of the equation, then it is obvious that $\tilde{y} \geq k_x$. Let $A = (\tilde{y} > k_x)$ and $P(A) >0$. Since $(\tilde{y}-x)^+ \geq (k_x-x)^+$ and $E[(\tilde{y}-x)^+] = E[z] = E[(k_x-x)^+]$, which implies $(\tilde{y}-x)^+ = (k_x-x)^+$. Since it is impossible that $k_x-x \leq 0$ on $A$ (otherwise, $0 = \int_A (k_x-x)^+dP = \int_A E[(k_x-x)^+~|~\mathcal{F}]dP = \int_A z dP$, which contradicts the fact that $z \in L^0_{++}(\mathcal{F}) \cap L^1(\mathcal{F})$), there exists some $B \in \mathcal{E}$ and $ B \subset A$ such that $P(B) >0$ and $(k_x-x)>0$ on $B$, then $\tilde{y} - x = k_x-x$ on $B$, which again contradicts $\tilde{y} > k_x$ on $A$.
\end{proof}

Since $\frac{1}{\beta} \in L^0_{++}(\mathcal{F}) \cap L^1(\mathcal{F})$, there exists unique one $k_x \in L^0(\mathcal{F})$ for each $x \in L^2_{\mathcal{F}}(\mathcal{E})$ such that $E[(k_x-x)^+~|~\mathcal{F}] = \frac{1}{\beta}$. Observe that $(k_x - x)^+ = k_x -k_x\wedge x$, $E[k_x\wedge x~|~\mathcal{F}] = k_x -\frac{1}{\beta}$, and it is also clear that $k_x\wedge x \in G_\beta$ for any $x \in L^2_{\mathcal{F}}(\mathcal{E})$ and it is also easy to see that $k_x = E[x~|~\mathcal{F}] + \frac{1}{\beta}$ and $k_x \wedge x =x$ for any $x \in G_\beta$.

From now on, $k_x$ always stands for the unique solution of $E[(k-x)^+~|~\mathcal{F}] = \frac{1}{\beta}$ for any given $\beta$ and $x \in L^2_{\mathcal{F}}(\mathcal{E})$.

\begin{theorem}\label{theorem3.10}
$V_\beta (x) = U_\beta (x \wedge k_x) = - E[x \wedge k_x~|~\mathcal{F}] + \frac{\beta}{2}D[x \wedge k_x~|~\mathcal{F}]$ for any $x \in L^2_{\mathcal{F}}(\mathcal{E})$. Further, $V_\beta$ is $L^0$--G$\hat{a}$teaux differentiable and $\mathcal{T}_c$--continuous on $L^2_{\mathcal{F}}(\mathcal{E})$ and $V'_\beta(x) = - \beta (k_x -x)^+$ for any $x \in L^2_{\mathcal{F}}(\mathcal{E})$.
\end{theorem}
\begin{proof}
Since $Dom(V_\beta) = L^2_{\mathcal{F}}(\mathcal{E})$ and $V_\beta$ is $\mathcal{T}_c$--lower semicontinuous, $V_\beta$ is $\mathcal{T}_c$--continuous (equivalently, $\mathcal{T}_{\varepsilon,\lambda}$--continuous, see \cite[Theorem 3.6]{GZWYYZ17}) and $L^0$--subdifferentiable by Theorems \ref{theorem2.29} and \ref{theorem2.30}. In fact, Theorem 5.7 of \cite{FKV12} already constructed an $L^0$--subgradient $-\beta(k_x -x)^+$ of $V_\beta$ at any $x \in L^2_{\mathcal{F}}(\mathcal{E})$ and Lemma 5.5 of \cite{FKV12} showed that $-\beta (k_x - x)^+$ maximizes the $L^0$--concave function $g:L^2_{\mathcal{F}}(\mathcal{E}) \rightarrow L^0(\mathcal{F})($ defined by $g(z)=E[xz|\mathcal{F}]-U^*_{\beta}(z)$ for each $z\in \mathcal{P}^0_c\cap \mathcal{D})$ over $\mathcal{P}^0_c\cap \mathcal{D}$. Thus
\begin{equation}\nonumber
\aligned
V_{\beta}(x)&=\bigvee\{E[xz|\mathcal{F}]-\frac{1}{2\beta} E[(1+z)^2|\mathcal{F}]: z\in \mathcal{P}^0_c\cap \mathcal{D}\}\\
&=E[x\cdot(-\beta(k_x-x)^+)|\mathcal{F}]- \frac{1}{2\beta} E[(1-\beta(k_x-x)^+)^2|\mathcal{F}]\\
&=-E[\beta(x-k_x)(k_x-x)^+|\mathcal{F}]- \frac{\beta}{2} E[((k_x-x)^+)^2|\mathcal{F}]+ \frac{1}{2\beta} -k_x\\
&=\beta E[((k_x-x)^+)^2|\mathcal{F}]-  \frac{\beta}{2} E[((k_x-x)^+)^2|\mathcal{F}]+ \frac{1}{2\beta} -k_x\\
&=\frac{\beta}{2} E[((k_x-x)^+)^2|\mathcal{F}]+ \frac{1}{2\beta} -k_x\\
&=\frac{\beta}{2} (E[((k_x-x)^+)^2|\mathcal{F}] - E[(k_x-x)^+|\mathcal{F}]^2) + \frac{\beta}{2}E[(k_x-x)^+|\mathcal{F}]^2+ \frac{1}{2\beta} -k_{x}\\
&=\frac{\beta}{2} D[(k_x-x)^+|\mathcal{F}]+ \frac{1}{2\beta} + \frac{1}{2\beta}-k_{x}\\
&=\frac{\beta}{2} D[k_x\wedge x|\mathcal{F}]- E[k_x\wedge x|\mathcal{F}].
\endaligned
\end{equation}
\par
Finally, since, for each fixed $x\in L^2_{\mathcal{F}}(\mathcal{E}), g$ defined as above is strictly $L^0$--concave on $\mathcal{P}^0_c\cap \mathcal{D}$, $-\beta(k_x-x)^+$, as a maximizer of $g$ over $\mathcal{P}^0_c\cap \mathcal{D}$, must be unique, namely $\partial V_{\beta}(x)=\{-\beta(k_x-x)^+\}$, then $V'_{\beta}(x)=-\beta(k_x-x)^+$.
\end{proof}

When $\mathcal{F}$ is trivial $($namely $\mathcal{F}=\{\Omega,\emptyset \})$, a conditional convex risk measure automatically reduces to a $($static$)$ convex risk measure. Example \ref{example3.11} below illustrates that entropic risk measure is not coercive on the standard Markowitz constraint set and Example \ref{example3.12} below shows that the same phenomenon also occurs for the monotone mean--variance risk measure.

\begin{example}\label{example3.11}
Let $\Omega=[0,1]$, $\mathcal{E}$ the $\sigma$--algebra of Lebesgue measurable subsets of $[0,1]$ and $P$ the Lebesgue measure on $\mathcal{E}$. For any given positive number $p$ such that $1<p<+\infty$, we can construct a sequence $\{x_n: x\in N\}$ in $L^p(\mathcal{E})$ such that
\begin{enumerate}[(1)]
\item $\pi(x_n)=E[x_n]=1$ for each $n\in N$, where $\pi(x)=E[x]=\int_{\Omega}xdP$ for any $x\in L^p(\mathcal{E})$, and hence $\pi$ can act as a pricing function;
\item $E(|x_n|^p)\rightarrow +\infty$ when $n$ tends to $+\infty$;
\item $\rho_{\gamma}(x_n)=\frac{1}{\gamma}\ln E[e^{-\gamma x_n}]\leq 1$ for each $n\in N$.
\end{enumerate}
\par
Define $A^-_n=(\frac{1}{2}, 1], A^+_n=((\frac{1}{2})^{n+1}, \frac{1}{2}], A^0_n=[0, (\frac{1}{2})^{n+1})$ and $x_n$ as follows for each $n\in N$;
$$ x_n(\omega)=\left\{
\begin{aligned}
0 & , & \omega\in A^0_n; \\
c_n \omega^{-\frac{1}{p}} & , & \omega\in A^+_n; \\
-1 & , & \omega\in A^{-}_n.
\end{aligned}
\right.
$$
\par
Here, $c_n$ is a normalized positive constant such that $E[x_n]=1$, then it is obvious that $c_n= \frac{\frac{3}{2}(1-\frac{1}{p})}{(\frac{1}{2})^{1-\frac{1}{p}}- (\frac{1}{2})^{(n+1)(1-\frac{1}{p})}}$. Clearly, $\{c_n: n\in N\}$ is a bounded sequence since $c_n\in [\frac{\frac{3}{2}(1-\frac{1}{p})}{(\frac{1}{2})^{(1-\frac{1}{p})}}, \frac{\frac{3}{2}(1-\frac{1}{p})}{(\frac{1}{2})^{(1-\frac{1}{p})}- (\frac{1}{4})^{(1-\frac{1}{p})}}]$ for each $n\in N$.
\par
It is easy to check that $\{x_n: n\in N\}$ is a sequence in $L^p(\mathcal{E})$, but $E[|x_n|^p]\geq E[I_{A^+_n}\cdot x^p_n]= c^p_n\int_{A^+_n} \frac{1}{\omega}d\omega=c^p_n\cdot n\cdot \ln2$, and hence $(2)$ is satisfied.
\par
Finally, $E[e^{-\gamma x_n}]= E[e^{-\gamma x_n}\cdot I_{A^+_n}]+ E[e^{-\gamma x_n}\cdot I_{A^0_n}]+ E[e^{-\gamma x_n}\cdot I_{A^-_n}]\leq E[1\cdot I_{A^+_n}]+ E[1\cdot I_{A^0_n}]+ E[e^{-\gamma x_n}\cdot I_{A^-_n}]\leq P(A^+_n)+ P(A^0_n)+ e^{\gamma}P(A^-_n)\leq e^{\gamma}$, and hence $(3)$ is also satisfied. $\square$
\end{example}
\par
First, let us again notice that $\{c_n: n\in N\}$ satisfies $\frac{3}{4}\sqrt{2}\leq c_n\leq \frac{3}{2}(\sqrt{2}+1)$ for any $n\in N$ when $p=2$. Example \ref{example3.12} below illustrates that for almost all positive numbers $\beta$ the monotone mean--variance risk measure $V_{\beta}: L^2(\mathcal{E})\rightarrow R$, defined by $V_{\beta}(x)=-E[k_x\wedge x]+ \frac{\beta}{2}D[k_x\wedge x]$ for any $x\in L^2(\mathcal{E})$, is not coercive on the constraint set $\{x\in L^2(\mathcal{E}) : \pi(x)=1$ and $E[x]=1\}$, where $(\Omega,\mathcal{E}, P)$ and $\pi: L^2(\mathcal{E})\rightarrow R$ are the same as in Example \ref{example3.11} $($when  $p=2)$, $k_x\in R$ is the unique solution of the equation $E[(k-x)^+]=\frac{1}{\beta}$ for each fixed $x\in L^2(\mathcal{E})$, $G_{\beta}=\{x\in L^2(\mathcal{E}): x-E[x]\leq \frac{1}{\beta}\}$ and $D[x]$ denotes  the variance of $x\in L^2(\mathcal{E})$.

\begin{example}\label{example3.12}
Let $p=2, (\Omega, \mathcal{E}, P)$ and the sequence $\{x_n: n\in N\}$ be the same as in Example \ref{example3.11}. Then, for the prescribed positive numbers $\beta$, $\{x_n: n\in N\}$  satisfies the following conditions:
\begin{enumerate}[(1)]
\item $\pi(x_n)=E[x_n]=1$ for each $n\in N$;
\item $E[|x_n|^2] \rightarrow +\infty$ when $n \rightarrow \infty$;
\item $\{V_{\beta}(x_n): n\in N\}$ is always bounded.
\end{enumerate}
Check. First, let us notice the known fact that from \cite{MMRT09}: $E[k_x\wedge x]=k_x-\frac{1}{\beta}$, further it is also obvious that $D[k_x\wedge x]=D[(k_x-x)^+]= E[((k_x-x)^+)^2]- \frac{1}{\beta^2}$. Besides, it is easy to see that $x_n \notin G_{\beta}$ for sufficiently large $n$, for simplicity we denote $k_{x_n}$ by $k_n$ for each $n\in N$. Now, we can illustrate the following three cases of $\beta$ for which $(1),(2)$ and $(3)$ hold.\\
Case 1: $\beta\geq 2$.
\par
A simple computation shows that it is impossible to find a positive number $k_n$ such that $E[(k_n-x_n)^+]= \frac{1}{\beta}$ for each $n\in N$. Let $k_n\leq 0$ be such that $E[(k_n-x_n)^+]= \frac{1}{\beta}$, then $\frac{(k_n+1)^+}{2}=\frac{1}{\beta}$, namely $\frac{k_n+1}{2}=\frac{1}{\beta}$, that is, $k_n=\frac{2}{\beta}-1$ for any $n\in N$. Then $V_{\beta}(x_n)= -E[k_n\wedge x_n]+ \frac{\beta}{2}D[k_n\wedge x_n]= \frac{1}{\beta}- k_n+ \frac{\beta}{2}(\frac{(k_n+1)^2}{2}- \frac{1}{\beta^2})= \frac{1}{\beta}- (\frac{2}{\beta}-1)+ \frac{\beta}{2}(2(\frac{k_n+1}{2})^2- \frac{1}{\beta^2})= 1- \frac{1}{2\beta}$.\\
Case 2: $\frac{4}{5}< \beta< 2$.
\par
It is very easy to see that $k_n$ must be positive for each $n\in N$. Let $k^*_n= \frac{\frac{1}{\beta}- \frac{1}{2}}{\frac{1}{2}+ (\frac{1}{2})^{n+1}}$ for each $n\in N$, then $k^*_n< \frac{2}{\beta}-1 < \frac{3}{2} \leq \sqrt{2}c_n$, and hence $(k^*_n- c_n\omega^{-\frac{1}{2}})^+=0$ for each $\omega\in ((\frac{1}{2})^{n+1}, \frac{1}{2}]$, from which it is easy to see that $E[(k^*_n-x_n)^+]= \frac{1}{\beta},$ namely $k_n= k^*_n$, and at this time $V_{\beta}(x_n)= \frac{1}{\beta}- k_n+ \frac{\beta}{2}(k^2_n\cdot(\frac{1}{2})^{n+1}+ \frac{(k_n+1)^2}{2}- \frac{1}{\beta^2})$, and hence $\{V_{\beta}(x_n): n\in N\}$ is bounded since $\{k_n: n\in N\}$ is bounded.\\
Case 3: $0< \beta< \frac{1}{\frac{3}{2}(\sqrt{5}-\sqrt{2})(\sqrt{2}+1)+ \frac{1}{2}    }(<\frac{4}{5})$.
\par
Since $\frac{1}{\beta}- \frac{1}{2}> \frac{3}{2}(\sqrt{5}-\sqrt{2})(\sqrt{2}+1)\geq (\sqrt{5}-\sqrt{2})c_n$, $\sqrt{2}c_n+ \frac{1}{\beta}- \frac{1}{2}> \sqrt{5}c_n\geq 2c_n\sqrt{1+\frac{1}{2^{n+1}}}$ for any $n\in N$. Let $k^*_n= \frac{a_n+ \sqrt{a^2_n- b^2_n}}{2(1+\frac{1}{2^{n+1}})   }$, where $a_n= \sqrt{2}c_n+ \frac{1}{\beta}- \frac{1}{2}$ and $b_n= 2c_n\sqrt{1+\frac{1}{2^{n+1}}}$. It is clear that $\{k^*_n: n\in N\}$ is bounded and $k^*_n\geq \sqrt{2}c_n$ for any $n\in N$, and hence $\frac{1}{2^{n+1}}\leq (\frac{c_n}{k^*_n})^2\leq \frac{1}{2}$ for sufficiently large $n$, then $E[(k^*_n-x_n)^+]= \frac{k^*_n+1}{2}+ k^*_n(\frac{1}{2})^{n+1}+ \int^{\frac{1}{2}}_{(\frac{c_n}{k^*_n})^2}(k^*_n-c_n\omega^{-\frac{1}{2}})d\omega= k^*_n(1+ \frac{1}{2^{n+1}}) -2c_n(\frac{1}{2})^{\frac{1}{2}}+ \frac{c^2_n}{k^*_n}+ \frac{1}{2},$ it is also very easy to check that $k^*_n$ is the unique solution of the equation: $k^2(1+\frac{1}{2^{n+1}})- (\sqrt{2}c_n+ \frac{1}{\beta}-  \frac{1}{2})k+ c^2_n=0$ such that $k\geq \sqrt{2}c_n,$ namely $E[(k^*_n-x_n)^+]=\frac{1}{\beta}$, that is to say, $k_n=k^*_n$ for sufficiently large $n$. Further, for such sufficiently large $n$, $V_{\beta}(x_n)= E[((k_n-x_n)^+)^2]- \frac{1}{\beta^2}+ \frac{1}{\beta}-k_n= \frac{(k_n+1)^2}{2}+ k^2_n\cdot (\frac{1}{2})^{n+1} + \int^{\frac{1}{2}}_{(\frac{c_n}{k_n})^2}(k_n-c_n\cdot \omega^{-\frac{1}{2}})^2d\omega + \frac{1}{\beta}-k_n\leq \frac{(k_n+1)^2}{2}+ k^2_n\cdot (\frac{1}{2})^{n+1}+  k^2_n\cdot (\frac{1}{2})+\frac{1}{\beta}-k_n. $ So $\{V_{\beta}(x_n): n\in N\}$ is bounded.
\end{example}

\section{Optimization of conditional convex risk measures.}\label{section4}
In this section, let $1 < p < +\infty$, $M$ a closed submodule of $L^p_{\mathcal{F}}(\mathcal{E})$, which stands for the set of payoffs, and $\pi : M \to L^0(\mathcal{F})$ a continuous module homomorphism from $(M,\mathcal{T}_{\varepsilon,\lambda})$ to $(L^0(\mathcal{F}),\mathcal{T}_{\varepsilon,\lambda})$, namely $\pi \in M^*$, which stands for a pricing function. Throughout this section, we always suppose $\pi$ and $M$ satisfies Assumption \ref{assumption4.1} below.
\begin{assumption}\label{assumption4.1}
$L^0(\mathcal{F}) \subset M$ and there exists $r^f \in L^0_{++}(\mathcal{F})$ such that $\pi (r^f) =1$ $($namely $r^f$ is a risk--free return$)$.
\end{assumption}

We also define two level sets of $\pi$ that are central to our analysis as follows:\\
$R_\pi = \{ x \in M : \pi (x) =1 \}$ \\
$Z_\pi = \{ x \in M : \pi (x) =0 \}$.\\
The set $R_\pi$ is called the set of returns and $Z_\pi$ the set of payoffs with zero price. Assumption \ref{assumption4.1} assures $R_\pi$ is nonempty, it is obvious that $R_\pi = r + Z_\pi$ for an arbitarily chosen $r \in R_\pi$.

\begin{remark}\label{remark4.2}
Assumption \ref{assumption4.1} was already employed in \cite{MMRT09} where F. Maccheroni et.al only considered the unconditional case (namely $\mathcal{F} = \{ \Omega, \emptyset \}$) and $R$ in \cite{MMRT09} amounts to $r^f$ here, and in particular F. Maccheroni \cite{MMRT09} only considered the finite-dimensional case when $M$ is the subspace spanned by the payoffs of the $n+1$ assets, the first $n$ assets of which are risky and the $(n+1)$--th one is risk--free. Here, we generalize the work of \cite{MMRT09} to the infinite--dimensional case and also allow $\mathcal{F}$ to be a generic $\sigma$--subalgebra (namely we consider the general conditional version of \cite{FKV12}).

\end{remark}

\begin{remark}\label{remark4.3}
Another assumption: there exists $z_0 \in Z_\pi$ such that $P \{ E[z_0~|~\mathcal{F}] =0 \} =0$, is often made, for example, in \cite{HR87}. This assumption and Assumption \ref{assumption4.1} together amounts to the fact: the submodule spanned by 1 and $\pi$, denoted by span$\{ 1,\pi \}$, is a free $L^0(\mathcal{F})$--submodule of $M^*$, where 1, regarded as an element of $M^*$, is identified with the restriction of the generalized expectation operator $E[\cdot~|~\mathcal{F}]$ to $M$, namely $E[1 \cdot x~|~\mathcal{F}]= E[x~|~\mathcal{F}]$ for any $x \in M$. Here, we say that span $\{ 1,\pi \}$ is $L^0(\mathcal{F})$--free, namely $\xi$ and $\eta$ in $L^0(\mathcal{F})$ are such that $\xi \cdot 1 + \eta \cdot \pi =0$ iff $\xi$ and $\eta$ are both zero.

\end{remark}

Lemma \ref{lemma4.4} below was proved by Guo in \cite{Guo97} in 1997, which has played a crucial role in the theory of random duality in \cite{GC09} as well as in this paper.

\begin{lemma}\cite{Guo97,GC09}\label{lemma4.4}
Let $E$ be an $L^0(\mathcal{F},K)$--module and $f_1,f_2,\cdots,f_n$ and $g$~$n+1$ module homomorphisms from $E$ to $L^0(\mathcal{F},K)$. Then $\bigcap^n_{i=1} N(f_i) \subset N(g)$ iff there exist $\xi_1,\xi_2,\cdots,\xi_n \in L^0(\mathcal{F},K)$ such that $g = \sum^n_{i=1} \xi_i \cdot f_i$, where $N(f)$ denotes the kernel space of a module homomorphism $f$ from $E$ to $L^0(\mathcal{F},K)$.
\end{lemma}

Now, let $V_\beta : L^2_{\mathcal{F}}(\mathcal{E}) \to L^0(\mathcal{F})$ be the conditional monotone mean--variance risk measure with the aversion index $\beta$ to conditional variance, as in Section \ref{section3}, we can consider solutions to the conditional monotone mean--variance problem.

\begin{problem}\label{problem4.5}
Minimize $V_\beta(x)$ for $x$ in $R_\pi$.
\end{problem}

Theorem \ref{theorem4.6} below provides a perfect solution to Problem \ref{problem4.5}, which is one of our central results of this section. Similar to the notion of a Banach lattice, one can have the notion of a complete $RN$ module lattice (or random Banach lattice in terms of \cite{HLR91}), it is clear that $L^p_{\mathcal{F}}(\mathcal{E})$ is a random Banach lattice.

\begin{theorem}\label{theorem4.6}
$x^* \in R_\pi$ is a solution to Problem \ref{problem4.5} iff $- \frac{1}{r^f} \cdot V'_\beta(x^*)|_M = \pi$. In particular, when $M$ is also a sublattice of $L^2_{\mathcal{F}}(\mathcal{E})$, namely, $|x| \in M$ for any $x \in M$, $x^* \wedge k_{x^*}$ is unique $($and hence $k_{x^*}$ is also unique$)$, at which time $V'_\beta (x^*)|_M = V'_\beta (x^*)$. Where $V'_\beta (x^*)|_M$ is the restriction of $V'_\beta (x^*)$ to $M$.
\end{theorem}

\begin{proof}
Let us recall that $V'_\beta (y) = - \beta (k_y -y)^+$ for any $y \in L^2_{\mathcal{F}}(\mathcal{E})$ and $R_\pi$ is a $\mathcal{T}_{\varepsilon,\lambda}$--closed $L^0$--convex subset of $M$. Since $x^* \in R_\pi$ is a solution to Problem \ref{problem4.5} iff $V'_\beta (x^*)(x-x^*) = \lim_{t \downarrow 0} \frac{V_\beta (x^* + t(x-x^*)) - V_\beta (x^*)}{t} \geq 0$ for any $x \in R_\pi$, namely, iff $V'_\beta (x^*)(h) \geq 0$ for any $h \in Z_\pi$ (please bear in mind that $R_\pi = x^* + Z_\pi$). Again observe that $Z_\pi = - Z_\pi$, then $x^*$ is a solution to Problem \ref{problem4.5} iff $V'_\beta (x^*)(h) =0$ for any $h \in Z_\pi$, and hence by Lemma \ref{lemma4.4} $V'_\beta (x^*)(h) =0$ $\forall h \in Z_\pi$ iff there exists $\xi \in L^0(\mathcal{F})$ such that $V'_\beta (x^*)|_M = \xi \cdot \pi$, then $V'_\beta (x^*)(r^f) = \xi \cdot \pi(r^f) = \xi$. Notice that $V'_\beta (x^*)(r^f) = E[- \beta (k_{x^*} - x^*)^+ \cdot r^f~|~\mathcal{F}] = r^f E[- \beta (k_{x^*}- x^*)^+~|~\mathcal{F}] = - r^f$, so $\xi =-r^f$, then $\pi = - \frac{1}{r^f} V'_\beta (x^*)|_M$.

Since $V'_\beta (x^*)|_M$ is exactly the orthogonal projection of $V'_\beta (x^*)$ to $M$, when $M$ is a sublattice of $L^2_{\mathcal{F}}(\mathcal{E})$, $V'_\beta (x^*) = -\beta (k_{x^*} - x^*)^+ \in M$, one can have $V'_\beta (x^*)|_M = V'_\beta (x^*)$. Let $y^* \in R_\pi$ be another solution to Problem \ref{problem4.5}, then $V'_\beta (x^*) = V'_\beta (y^*)$, namely $(k_{y^*} - y^*)^+ = (k_{x^*} - x^*)^+$, which in turn implies that $D[k_{y^*} \wedge y^*~|~\mathcal{F}] = D[(k_{y^*} -y^*)^+~|~\mathcal{F}] = D[(k_{x^*} -x^*)^+~|~\mathcal{F}] = D[k_{x^*} \wedge x^*~|~\mathcal{F}]$. If follows from $V_\beta (x^*) = V_\beta (y^*)$ that $E[k_{x^*} \wedge x^*~|~\mathcal{F}] = E[k_{y^*} \wedge y^*~|~\mathcal{F}]$, namely $k_{x^*} -\frac{1}{\beta}= k_{y^*} -\frac{1}{\beta}$, that is, $k_{x^*} = k_{y^*}$. It immediately follows from $(k_{x^*} - x^*)^+ = (k_{y^*} - y^*)^+$ that $k_{x^*} \wedge x^* = k_{y^*} \wedge y^*$.

\end{proof}

\begin{remark}\label{remark4.7}
Theorem \ref{theorem4.6} also produces an arbitrage free pricing rule: $\pi(y) = E[-\frac{1}{r^f} \cdot (-\beta (k_{x^*} - x^*)^+)\cdot y~|~\mathcal{F}] = \frac{1}{r^f} E[\beta (k_{x^*} - x^*)^+ \cdot y~|~\mathcal{F}]$ for any $y \in M$. Write $k^\beta_{x^*}$ for the $k_{x^*}$ and $x^*_\beta$ for the $x^*$, then it is obvious that $\beta(k^\beta_{x^*_\beta} - x^*_\beta)^+$ does not depend on $\beta$, denoted by $\triangledown V$, then $\pi(y) = \frac{1}{r^f} E[(\triangledown V) \cdot y~|~\mathcal{F}]$ for any $y \in M$. According to the Riesz representation theorem in complete $RIP$ modules \cite{GY96}, there exists unique one $\pi^* \in M$ such that $\pi (y) = E[\pi^* \cdot y~|~\mathcal{F}]$ for any $y \in M$, then it is clear that $\pi^*$ is exactly the orthogonal projection of $\frac{1}{r^f} \cdot \triangledown V$ to $M$. Finally, since $\triangledown V$ is nonnegative, $\pi$ is a nonnegative $L^0$--linear functional, in fact, this also shows that Theorem \ref{theorem4.6} essentially gives a representation for $\pi$ by $\frac{1}{r^f} \triangledown V$. By the way, let us link our Theorem \ref{theorem4.6} here with Theorem 4.1 of \cite{MMRT09}. In \cite{MMRT09}, the $n+1$ assets are considered: the first $n$ assets are risky, while the $(n+1)$th asset is risk--free. F.Maccheroni, et.al \cite{MMRT09} only consider the one--period (namely nonconditional or static) allocation problem, the gross return on the $i$th asset after one period is denoted by $X_i$ $(X = (X_1,X_2,\cdots,X_n)')$ (here $'$ stands for the usual transpose operator) and the return on the $(n+1)$th asset is risk--free and equal to $R$ (a positive number). Again, assume that $Y_i$ be the end--of--period payoff per share of asset $i$ and $p_i$ its current price for $i = 1,2,\cdots,n+1$, then $X_i = Y_i/p_i$ for $i = 1,2,\cdots,n$ and $R = Y_{n+1}/p_{n+1}$. For a portfolio consisting of $q_i$ shares of asset $i$, the corresponding payoff is $\sum_{i=1}^{n+1} q_i Y_i$, whose price is $\sum_{i=1}^{n+1} q_i p_i$, namely this is equivalent to defining a pricing function $\pi : M= \{ \sum_{i=1}^{n+1} q_i Y_i~|~(q_1,q_2,\cdots,q_{n+1})' \in R^{n+1} \} \to (-\infty, +\infty)$ by $\pi (\sum_{i=1}^{n+1} q_i Y_i) = \sum_{i=1}^{n+1} q_i p_i$, especially $\pi(X_i) =1$ for $i=1,2,\cdots,n$ and $\pi(R)=1$, at which time $R_\pi =\{ y \in M : \pi(y)=1 \}$, which can be written as $R_\pi = \{ R + \alpha \cdot (X - R \cdot \vec{1}) : \alpha \in R^n \}$ (see \cite{MMRT09}). It seems to me that to obtain their Theorem 4.1 F.Maccheroni, et.al implicitly assume in \cite{MMRT09} a very restrictive condition: $D[X~|~A]$ is positive definite for each $A \in \mathcal{F}$ with $P(A) >0$, where $D[X~|~A]$ stands for the conditional covariance matrix of $X$ on the event $A$. Considering the uniqueness problem, our Theorem \ref{theorem4.6} requires $M$ is a Banach sublattice of $L^2(\mathcal{E})$, it is clear that $M$ is rarely finite dimensional. In fact, if $M$ is a Banach sublattice of $L^2(\mathcal{E})$, then it follows from Corollary 1.b.4 of \cite{LT79} that $M$ is order isometric to $L^2(\Omega',\mathcal{E}',P')$ for a suitable probability space $(\Omega',\mathcal{E}',P')$, we can assume, without loss of generality, that $(\Omega',\mathcal{E}',P')$ is just $(\Omega,\mathcal{E},P)$. Thus $M$ is finite dimensional, $\mathcal{E}$ is necessarily generated by finitely many (for example, $m$) $v$--atoms. Further, if $D[X]$ is positive definite, then $\{ 1,X_1,X_2,\cdots,X_n \}$ is linearly independent, and hence $m$ must be greater than or equal to $n+1$. Careful readers can easily check that all the examples provided in \cite{MMRT09} satisfy the condition that $\mathcal{E}$ is purely atomic and $m \geq n+1$, at which time it is impossible that $D[X~|~A]$ is positive definite for any $A \in \mathcal{E}$ with $P(A) >0$. Thus, concerning the uniqueness problem, our Theorem \ref{theorem4.6} at least improves the formulation of Theorem 4.1 of \cite{MMRT09}. According to the above--stated discussions, we suggest that the space $M$ of payoffs should be an infinite dimensional Banach sublattice of $L^2(\mathcal{E})$ when people consider the optimization problem of monotone mean--variance for a generic probability space $(\Omega,\mathcal{E},P)$!
\end{remark}

Now, we consider the more general conditional mean--conditional convex risk measure (briefly, conditional mean--risk measure) problem. Let us first make a nontrivial assumption--Assumption \ref{assumption4.8} below.

\begin{assumption}\label{assumption4.8}
There exists some $z_0 \in Z_\pi$ such that $P \{ E[z_0~|~\mathcal{F}] =0 \} =0$.
\end{assumption}

\begin{problem}\label{problem4.9}
Minimize $\rho (x)$ for $x$ in $R_\pi$ subject to the constraint that $E[x~|~\mathcal{F}] = \mathrm{w}$ and $|||x|||_p \leq r$ for any fixed $\mathrm{w} \in L^0(\mathcal{F})$ and $r \in L^0_{++}(\mathcal{F})$, where $\rho : L^p_{\mathcal{F}}(\mathcal{E}) \to \bar{L}^0(\mathcal{F})$ is a given conditional convex risk measure.
\end{problem}

\begin{remark}\label{remark4.10}
We only consider the meaningful case: $G := \{ x \in M : \pi (x) =1, E[x~|~\mathcal{F}] = \mathrm{w}$ and $|||x|||_p \leq r \}$ is not vacuous, for example, this is the case when $r$ is sufficiently large. It is also obvious that $G$ is an almost surely (a.s.) bounded $\mathcal{T}_{\varepsilon,\lambda}$--closed $L^0$--convex subset of $L^p_{\mathcal{F}}(\mathcal{E})$.
\end{remark}

Let $(E,\|\cdot\|)$ be an $RN$ module over $K$ with base $(\Omega,\mathcal{F},P)$, a subset $H$ of $E$ is said to be a.s. bounded if there exists some $\xi \in L^0_+(\mathcal{F})$ such that $\|x\| \leq \xi$ for any $x \in H$, namely, $\bigvee\{\|x\|: x\in H\}\in L^0_+(\mathcal{F})$.
\par
On Problem \ref{problem4.9}, we have the following two results--Theorems \ref{theorem4.11} and \ref{theorem4.12}.
\begin{theorem}\label{theorem4.11}
Problem \ref{problem4.9} always has a solution $x^*$ in $G=\{x\in M: \pi(x)=1, E[x|\mathcal{F}] = \mathrm{w}$ and $|||x|||_p\leq r\}$.
\end{theorem}

\begin{theorem}\label{theorem4.12}
When $\rho$ is the conditional entropic risk measure $\rho_{\gamma}$ $($see Example \ref{example3.5}$)$, $x^*$ is also unique.
\end{theorem}
\par
Theorem \ref{theorem4.11} is based on the recently developed theory of $L^0$--convex compactness and closely related random convex optimization in \cite{GZWW20}, let us first recall some necessary concepts and basic results from \cite{GZWW20}.
\begin{definition}\label{definition4.13}\cite{GZWW20}
Let $(E,\mathcal{T})$ be a Hausdorff topological module over the topological algebra $(L^0(\mathcal{F},K),\mathcal{T}_{\varepsilon,\lambda})$. A nonempty closed $L^0$--convex subset $G$ of $E$ is said to have $L^0$--convex compactness $($or, to be $L^0$--convexly compact$)$ if any nonempty family of nonempty closed $L^0$--convex subsets of $G$ has a nonempty intersection whenever this family has finite intersection property.
\end{definition}

\begin{proposition}\label{proposition4.14}\cite{GZWW20}
Let $G$ be a nonempty $\mathcal{T}_{\varepsilon,\lambda}$--closed $L^0$--convex subset of a complete $RN$ module $(E,\|\cdot\|)$ over $K$ with base $(\Omega,\mathcal{F}, P)$. Then $G$ is $L^0$--convexly compact iff there exists some $x_0\in G$ for any given $f\in E^*$ such that $Re(f(x_0))=\bigvee\{Re(f(x)): x\in G\}$.
\end{proposition}

\begin{corollary}\label{corollary4.15}\cite{GZWW20}
A complete $RN$ module $(E,\|\cdot\|)$ is random reflexive iff every $\mathcal{T}_{\varepsilon,\lambda}$--closed and a.s. bounded $L^0$--convex subset of $E$ is $L^0$--convexly compact.
\end{corollary}

\par
Since a proper $L^0$--convex function $f$ on an $L^0$--convex and $\mathcal{T}_{\varepsilon,\lambda}$--closed subset $G$ of an $RLC$ module over $R$ with base $(\Omega,\mathcal{F}, P)$ must be stable, namely $f(\tilde{I}_A x+ \tilde{I}_{A^c}y)= \tilde{I}_Af(x)+ \tilde{I}_{A^c}f(y)$ for any $A\in \mathcal{F}$ and $x,y\in G$, Proposition \ref{proposition4.16} below, which can be called the fundamental theorem of random convex optimization, surveys an important special case of Theorem 3.6 and 3.8 of \cite{GZWW20}.

\begin{proposition}\label{proposition4.16}
Let $(E,\mathcal{P})$ be an $RLC$ module over $R$ with base  $(\Omega,\mathcal{F}, P)$ such that $E$ is $\sigma$--stable, $G$ a nonempty $\mathcal{T}_{\varepsilon,\lambda}$--closed $L^0$--convex subset of $E$ and $f: G\rightarrow \bar{L}^0(\mathcal{F})$ a proper $\mathcal{T}_c$--lower semicontinuous $L^0$--convex function. Then we have the following statements:
\begin{enumerate}[(1)]
\item If $G$ is $L^0$--convexly compact, then there exists some $x_0\in G$ such that $f(x_0)= \bigwedge \{f(x):x\in G  \}$.
\item If $E$ is a random reflexive $RN$ module and $f$ is random coercive, then there exists some $x_0\in G$ such that $f(x_0)= \bigwedge \{f(x):x\in G  \}$.
\end{enumerate}
\end{proposition}
\par
Now, we can prove Theorem \ref{theorem4.11}.
\begin{proof}[Proof of Theorem \ref{theorem4.11}]
Since $L^p_{\mathcal{F}}(\mathcal{E})$ is random reflexive $RN$ module and $G$ is a.s. bounded, $L^0$--convex and $\mathcal{T}_{\varepsilon,\lambda}$--closed, it immediately follows from $(1)$ of Proposition \ref{proposition4.16} and Corollary \ref{corollary4.15}.

\end{proof}

\begin{remark}\label{remark4.17}
Generally, one always hope that the constraint set $G$ is taken as $\{x\in M: E[x|\mathcal{F}]=\mathrm{w} $ and $\pi(x)=1\}$, for example, Hansen and Richard \cite{HR87} just did so for conditional variance since conditional variance is random strictly convex and random coercive on $\{x\in M: E[x|\mathcal{F}] =\mathrm{w}$ and $\pi(x)=1 \}$. Just as pointed out in \cite{GZWW20}, Lemma3.3 of \cite{HR87} is practically a special case of $(2)$ of Proposition \ref{proposition4.16}. However, since a conditional convex risk measure is neither random strictly convex nor random coercive on the constraint set $\{x\in M: \pi(x)=1$ and $E[x|\mathcal{F}]=\mathrm{w}\}$, as illustrated in Section \ref{section3}, we are forced to consider the constraint set $G =\{x\in M: \pi(x)=1, E[x|\mathcal{F}]=\mathrm{w}$ and $|||x|||_p\leq r\}$ in order to guarantee that Problem \ref{problem4.9} has at least one solution! We have not yet obtained a good theory of existence of Problem \ref{problem4.9}, it is fortunate for us to obtain Theorem \ref{theorem4.12}, whose proof needs Lemma \ref{lemma4.18} below.
\end{remark}

\par
To prove Lemma \ref{lemma4.18} below, let us first recall some convenient notations. Given $\xi$ in $L^0(\mathcal{F},K)$ with an arbitarily chosen representative $\xi^0: \Omega\rightarrow K,$ let $\theta(\omega)$ stand for the principal argument of $\xi^0(\omega)$ with value in $[0, 2\pi)$, we use $arg(\xi)$ for the equivalence class of $\theta$ and $e^{-iarg(\xi)}$ for the equivalence class of $e^{-i\theta}$. Besides, let $(E,\|\cdot\|)$ be an $RN$ module over $K$ with base $(\Omega, \mathcal{F},P)$ and $x\in E$, $A_x$ denotes the set $\{\omega\in \Omega: \|x\|^0(\omega)>0\}$ for any arbitarily chosen representative $\|x\|^0$ of $\|x\|$, which would produce no confusion since $A_x$ only differs by a null set for different choices of $\|x\|^0$.

\begin{lemma}\label{lemma4.18}
Let $(E,(\cdot, \cdot))$ be an $RIP$ module over $K$ with base $(\Omega, \mathcal{F},P)$ and $x,y$ in $E$ such that $|(x,y)|=\|x\|\cdot \|y\|$  and $\|y\|\in L^0_{++}(\mathcal{F})$, then there exists $\alpha\in L^0(\mathcal{F},K)$ such that $x=\alpha y$.
\end{lemma}

\begin{proof}
Let $arg((x,y))=\xi$, then $|(x,y)|=e^{-iarg(\xi)}\cdot (x,y)=(e^{-iarg(\xi)}x,y)$, so we can, without loss of generality, assume that $(x,y)=\|x\|\cdot\|y\|$, then $\|x+y\|^2=\|x\|^2+ 2\|x\|\cdot\|y\|+ \|y\|^2=(\|x\|+\|y\|)^2$, namely, $\|x+y\|=\|x\|+ \|y\|$. Since $(E,(\cdot, \cdot))$ is random strictly convex $($see \cite[Theorem3.3]{GZ10}$)$, when $x$ is not zero, by Definition 3.2 of \cite{GZ10} there exists $\xi\in L^0_+(\mathcal{F})$ such that $\xi>0$ on $A_{xy}$ and $\tilde{I}_{A_{xy}}x= \xi (\tilde{I}_{A_{xy}}y)$, where $A_{xy}=A_x\cap A_y$. Since $\|y\|\in L^0_{++}(\mathcal{F})$, $A_y=\Omega$ and $A_{xy}=A_x$, further let $\alpha=\tilde{I}_{A_x}\cdot \xi+ \tilde{I}_{A^c_x}\cdot 0$, then $x=\alpha y$. Finally, when $x=\theta$ $($the null of $E)$, taking $\alpha=0$ yields $x=\alpha y$.
\end{proof}
\par
We can now prove Theorem \ref{theorem4.12}.

\begin{proof}[Proof of Theorem \ref{theorem4.12}]
Let $u$ and $v$ in $G$ be two solutions to Problem \ref{problem4.9} for $\rho= \rho_{\gamma}$, then $\rho_{\gamma}(\frac{u+v}{2})=\frac{1}{2}\rho_{\gamma}(u)+  \frac{1}{2}\rho_{\gamma}(v) $, namely $\frac{1}{\gamma}\ln E[e^{-\gamma\cdot \frac{u+v}{2} }|\mathcal{F}]= \frac{1}{2}\gamma \ln E[e^{-\gamma u}| \mathcal{F}]+ \frac{1}{2}\gamma  \ln E[e^{-\gamma v}| \mathcal{F}]$, that is, $E[e^{-\frac{\gamma u}{2}}\cdot e^{-\frac{\gamma v}{2}}|\mathcal{F} ]= E[e^{-\gamma u}|\mathcal{F}]^{\frac{1}{2}}\cdot E[e^{-\gamma v}|\mathcal{F}]^{\frac{1}{2}}$.
\par
Now, let $x=e^{-\frac{\gamma u}{2}}$ and $y= e^{-\frac{\gamma v}{2}}$, then $x$ and $y$ are in $L^2_{\mathcal{F}}(\mathcal{E})$ and $(x,y)= E[xy|\mathcal{F}]= |||x|||_2\cdot |||y|||_2$. By Lemma \ref{lemma4.18} there exists $\alpha\in L^0(\mathcal{F})$ such that $x=\alpha y$, it is very easy to see that $\alpha\in L^0_{++}(\mathcal{F})$ and $-\frac{\gamma u}{2}= \ln \alpha- \frac{\gamma v}{2}$, namely $v-u= \frac{2}{\gamma}\ln \alpha$, which yields $\frac{2}{\gamma}\ln \alpha = E[(v-u)| \mathcal{F}]=0,$ and hence $\alpha=1$, that is, $x=y$.
\end{proof}

\section*{Acknowledgement}
This paper is supported by the NNSF of China (Grant No.11571369 and No.11971483). The author would like to thank his recent coauthors: Ph.D Minzhi Wu, Erxin Zhang, Yachao Wang and Professor George Yuan for their kind cooperation in recent years, which makes the author finish this paper.


\bibliography{amsplain}

\end{document}